\documentclass[final]{siamart250211}

\usepackage{amsfonts,amsmath,amssymb,bbm}
\usepackage{mathrsfs,mathtools,stmaryrd,wasysym}
\usepackage{xspace,mydef}
\usepackage{esint}
\usepackage{comment}
\usepackage[shortlabels]{enumitem}
\usepackage{tensor}
\DeclareMathAlphabet{\mathpzc}{OT1}{pzc}{m}{it}

\crefname{enumi}{assumption}{assumptions}
\Crefname{enumi}{Assumption}{Assumptions}

\newcommand{\TheTitle}{A pointwise tracking optimal control problem for a fractional, semilinear PDE}
\newcommand{\ShortTitle}{A pointwise tracking optimal control problem}
\newcommand{\TheAuthors}{E.Ot\'arola, A.J.~Salgado}

\headers{\ShortTitle}{\TheAuthors}

\title{\TheTitle}

\author{
  Enrique Ot\'arola\thanks{Departamento de Matem\'atica, Universidad T\'ecnica Federico Santa Mar\'ia, Valpara\'iso, Chile. (\email{enrique.otarola@usm.cl}, \url{http://eotarola.mat.utfsm.cl/})}
\and
  Abner J.~Salgado\thanks{Department of Mathematics, University of Tennessee, Knoxville, TN 37996, USA.
    (\email{asalgad1@utk.edu}, \url{https://math.utk.edu/people/abner-salgado/})}
}

\ifpdf
\hypersetup{
  pdftitle={\TheTitle},
  pdfauthor={\TheAuthors}
}
\fi

\begin{document}

\maketitle


\begin{abstract}
We analyze an optimal control problem with pointwise tracking for a fractional semilinear elliptic partial differential equation. The diffusion is characterized by the spectral fractional Laplacian $(-\Delta)^s$ with $s \in (1/2,1)$, a range that guarantees the well-posedness of point evaluations of the state. In addition to the nonconvexity of the control problem, the main difficulty is that the adjoint equation is a fractional partial differential equation with a singular right-hand side: a linear combination of Dirac measures. We establish the existence of optimal solutions and derive first-order as well as necessary and sufficient second-order optimality conditions.
\end{abstract}

\begin{keywords}
Optimal control problem, fractional diffusion, nonlocality, spectral fractional Laplacian, pointwise tracking, singular forces, Dirac measures, existence results, optimality conditions.
\end{keywords}

\begin{MSCcodes}
  35R06,    
  35R11,    
  49J20,    
  49K20.    
\end{MSCcodes}

\section{Introduction}
\label{sec:Introduction}

The aim of this paper is to study an optimal control problem with pointwise tracking for a fractional, semilinear, elliptic partial differential equation (PDE) that governs the state variable. Specifically, let $\Omega \subset \Real^2$ be a bounded, convex polygon, $s \in (\tfrac12,1)$, and $\Laps$ the fractional Laplace operator in the sense of spectral theory. In addition, let $\mathcal{D} \subset \Omega$ be a finite set of observable points, $\{ u_\vertex \}_{\vertex \in \mathcal{D}} \subset \mathbb{R}$ a set of desired states, and $\alpha >0$ a regularization parameter. With these ingredients, we introduce the cost functional
\begin{equation}
\label{eq:cost_functional}
 J(u,q) \coloneqq \frac{1}{2} \sum_{\vertex \in \mathcal{D}} | u(\vertex) - u_\vertex |^2 + \frac{\alpha}{2} \| q \|^2_{L^2(\Omega)}.
\end{equation}
Given a function $f$ and control bounds $a,b \in \mathbb{R}$, with $-\infty < a < b < \infty$, the pointwise tracking optimal control problem is as follows: Find $\min J(u,q)$ subject to the fractional, semilinear, elliptic PDE
\begin{equation}
  (-\Delta)^s u + \fraka(\cdot,u)= f + q \text{ in } \Omega,
  \label{eq:state_equation_strong}
\end{equation}
and the control constraints $q \in \mathbb{Q}_{ad}$, where
\begin{equation}
\label{eq:DefOfQ}
  \mathbb{Q}_{ad} \coloneqq \left\{ v \in L^2(\Omega): a \leq v(x) \leq b~\mae~x \in \Omega \right\}.
\end{equation}
Assumptions on the nonlinear function $\fraka$ will be deferred until \Cref{sub:Assumptions_on_a}.

The pointwise tracking optimal control problem described above is generally nonconvex because the state equation is nonlinear. Consequently, first-order optimality conditions are not sufficient, and a comprehensive study of this optimization problem requires examining second-order optimality conditions. Another main difficulty in analyzing this pointwise tracking optimal control problem is that the so-called \emph{adjoint problem} is a fractional PDE with a linear combination of Dirac measures on the right-hand side, namely:
\[
 (-\Delta)^s p + \frac{\partial \fraka}{\partial u} (\cdot,u) = \sum_{\vertex \in \mathcal{D}} (u(\vertex) - u_\vertex) \delta_\vertex \quad \text{in } \Omega.
\]
Here, $\delta_\vertex$ denotes the Dirac delta supported at the interior observation point $\vertex \in \mathcal{D}$.

Several works have studied pointwise tracking optimal control problems governed by integer-order PDEs. To our knowledge, the first treatment of a pointwise tracking optimal control problem for the Poisson problem is \cite{MR3449612}, which was later complemented and extended in \cite{MR3523574,MR3800041,MR3973329,MR4957570}. The semilinear case (problem \cref{eq:state_equation_strong} with $s = 1$) is analyzed in \cite{MR4438718}; because the control problem is generally non-convex in this setting, second order optimality conditions are also established in \cite{MR4438718}. Related results for fluid-flow models include pointwise tracking control of the Stokes system \cite{MR4304887,MR4218114} and the Navier--Stokes system \cite{MR5005578}. More recent contributions address a multiobjective pointwise tracking problem \cite{fuica2025errorestimatesmultiobjectiveoptimal} and a pointwise tracking optimal control problem for a Cahn–Hilliard–Navier–Stokes system \cite{dharmatti2026pointwisetrackingoptimalcontrol}.

In contrast to these advances, and to the best of our knowledge, this is the first paper to address the pointwise tracking optimal control problem described above for the fractional semilinear elliptic PDE \cref{eq:state_equation_strong}. Most available works consider cost functionals such as:
\begin{equation*}
 \mathcal{J}_1(u,q) = \frac{1}{2}  \| u - u_d \|_{L^2(\Omega)}^2 + \frac{\alpha}{2} \| q \|^2_{L^2(\Omega)},
 \quad
 \mathcal{J}_2(u,q) = \int_{\Omega}  L(x,u(x))\mathrm{d}x + \frac{\alpha}{2} \| q \|^2_{L^2(\Omega)},
\end{equation*}
where $L: \Omega \times \mathbb{R} \rightarrow \mathbb{R}$ is a suitable Carath\'eodory function and $u_d \in L^2(\Omega)$ is a desired state.
Focusing on optimal control problems involving the  \emph{spectral} fractional Laplacian, we note the following advances: the linear-quadratic setting with cost functional $\mathcal{J}_1$ has been studied in \cite{MR3429730,MR4015150}, while the semilinear setting with cost functional $\mathcal{J}_2$ has been addressed in \cite{MR4055455}. For problems involving the \emph{integral} fractional Laplacian, see \cite{MR3990191,MR4667825} for the linear-quadratic case with cost functional $\mathcal{J}_1$ and \cite{MR4358465,MR4055455} for the semilinear case with cost functional $\mathcal{J}_2$. We conclude this paragraph by mentioning the advances in parameter identification for nonlocal and fractional operators in \cite{MR3472639,MR4369059,MR4441219}, as well as results on optimal design for nonlocal models \cite{MR3339075,MR4989257}, optimal control of peridynamics models \cite{MR4629861}, optimal control of nonlocal operators \cite{MR4727532,MR4793106}, external optimal control \cite{MR3988258}, and optimal control with state constraints \cite{MR4119494}.

We organize our presentation as follows. In \Cref{sec:Notation}, we establish notation, define the spectral fractional Laplacian, and present suitable embedding results, alongside the main assumptions on the nonlinear function $\mathfrak{a}$ under which we operate. \Cref{sec:Fractional_PDEs_with_measures} provides the necessary framework to analyze the adjoint problem by extending the analysis of \cite{Fractional_Delta} to fractional linear PDEs with lower order terms and measure-valued right-hand sides. In \Cref{sec:Fractional_semilinear_PDEs}, we review and extend existence and regularity results for fractional semilinear PDEs. These results provide the foundation for \Cref{sec:Control_problem}, where we study the pointwise tracking optimal control problem of interest. Specifically, we prove the existence of at least one global optimal solution and analyze the differentiability of the control-to-state map. Finally, by introducing the adjoint problem --- a fractional linear PDE with a linear combination of Dirac measures --- we derive first order necessary optimality conditions, as well as necessary and sufficient second order optimality conditions.

\section{Notation}
\label{sec:Notation}

Let us introduce some relations that we will use in our work. $A \coloneqq B$ denotes equality by definition. $C \eqqcolon D$ stands for $D \coloneqq C$. $A \lesssim B$ means $A \leq c B$ for a nonessential constant $c$ that may change at each occurrence. $A \gtrsim B$ means $B \lesssim A$. Finally, $A \eqsim B$ is the short form for $A \lesssim B \lesssim A$.

Throughout the text, $\Omega \subset \Real^2$ is a bounded, convex polygon and $s \in (\tfrac12,1)$. We use standard notation for classical Lebesgue and Sobolev spaces. The space of finite Radon measures on $\Omega$ is denoted by $\calM(\Omega)$; see \cite[Definition 1.9]{MR3409135}. The duality pairing between $\calM(\Omega)$ and $C_0(\bar\Omega)$ --- the space of continuous functions in $\bar \Omega$ vanishing on $\partial \Omega$ --- will be denoted by $\langle\cdot,\cdot\rangle$.

\subsection{The spectral fractional Laplacian}
\label{sub:FracLaplace}

To introduce the spectral definition of the fractional Laplace operator \cite{MR2646117,MR2825595,MR2754080,MR3348172}, we present the eigenvalue problem:
\begin{equation}
\label{eq:LapEigenPairs}
(\lambda,\varphi) \in \Real \times H_0^1(\Omega) \setminus \{ 0\}:
\quad
  (\nabla \varphi, \nabla v)_{L^2(\Omega)} = \lambda (\varphi,v)_{L^2(\Omega)}
  \quad
  \forall v \in H_0^1(\Omega),
\end{equation}
which has a countable collection of solutions $\{(\lambda_k,\varphi_k)\}_{k=1}^\infty \subset \Real^+ \times \Hunz$ such that $\{\varphi_k\}_{k=1}^\infty$ is an orthonormal basis of $L^2(\Omega)$ and an orthogonal basis of $\Hunz$ \cite{MR609148}.

For $r \geq 0$, we define, in terms of the sequence of eigenpairs $\{(\lambda_k,\varphi_k)\}_{k=1}^\infty$,
\begin{equation}
\label{eq:DefOfPolHr}
  \polH^r(\Omega) \coloneqq \left\{ w = \sum_{k=1}^\infty w_k \varphi_k \ : \
  \| w\|_{\polH^r(\Omega)} < \infty \right \},
  \, \, \,
  \| w\|_{\polH^r(\Omega)} \coloneqq \left( \sum_{k=1}^\infty \lambda_k^r |w_k|^2 \right)^{\frac{1}{2}}.
\end{equation}
For $r >0$, $\polH^{-r}(\Omega)$ is the dual space of $\polH^{r}(\Omega)$. The duality pairing between $\mathbb{H}^{-r}(\Omega)$ and $\mathbb{H}^{r}(\Omega)$ is denoted by $_{-r}\langle \cdot, \cdot \rangle_{r}$. With $_{-r}\langle \cdot, \cdot \rangle_{r}$, we can extend the definition of the norm $\| \cdot \|_{\polH^r(\Omega)}$ to negative values of $r$. In fact, we can identify an element $F$ of $\mathbb{H}^{-r}(\Omega)$ with a sequence $\{ F_k \}_{k=1}^{\infty}$ such that
\[
 \| F \|_{\polH^{-r}(\Omega)} \coloneqq
 \left(
 \sum_{k=1}^\infty \lambda_k^{-r}|F_k|^2
 \right)^{\frac{1}{2}}
 < \infty.
\]

For $s \in (0,1)$ and $w \in C_0^\infty(\Omega)$, we define the \emph{spectral fractional Laplacian} as
\begin{equation}
\label{eq:DeofOfLaps}
  \Laps w \coloneqq \sum_{k=1}^\infty \lambda_k^s w_k \varphi_k, \qquad w_k \coloneqq \int_\Omega w \varphi_k \diff x.
\end{equation}
The operator $\Laps$ can be extended to $\mathbb{H}^s(\Omega)$ by density: $\Laps: \mathbb{H}^s(\Omega) \rightarrow \mathbb{H}^{-s}(\Omega)$. We note that $\Laps$ is an isomorphism between $\polH^s(\Omega)$ and its dual space $\polH^{-s}(\Omega)$.

We conclude this section with the following characterization of the spaces $\polH^r(\Omega)$ and certain embeddings.

\begin{proposition}[characterization of $\mathbb{H}^r(\Omega)$ for $0 \leq r \leq 2$]
\label{pro:characterization}
We have that
\[
   \polH^r(\Omega) = \begin{dcases}
                       H^r(\Omega), & r \in \left[ 0,\tfrac12 \right),\\
                       H^{\frac{1}{2}}_{00}(\Omega), & r = \tfrac12, \\
                       H_0^r(\Omega), & r \in \left(\tfrac12,1\right],
                     \end{dcases}
 \]
with equivalent norms. Moreover, if $r \in (1,2]$, then $\polH^r(\Omega) = H_0^1(\Omega) \cap H^r(\Omega)$, with equivalent norms. Consequently, for $r \in (0,1)$
\[
  \polH^r(\Omega) \hookrightarrow L^p(\Omega) \qquad \forall p \in \left[ 1, \frac2{1-r} \right)
\]
compactly and, if $r>1$, $\polH^r(\Omega) \hookrightarrow C(\bar\Omega)$ compactly.
\end{proposition}
\begin{proof}
For details on the characterizations, see \cite{MR350177,MR1742312,MR2328004,MR216336,MR3343061,MR3356020}. We then use the embedding properties of standard Sobolev spaces to conclude the claimed continuous embeddings (see \cite[Theorem 6.7]{MR2944369} and \cite[Theorem 4.12, \textbf{PART II}]{MR2424078}). The compact embedding for $r<1$ can be found in \cite[Corollary 7.2]{MR2944369}, while that for $r>1$ follows from \cite[Theorem 7.37]{MR2424078} and \cite[Theorem 1.34]{MR2424078} together with $B^r_{2,2}(\Omega) = H^r(\Omega)$, which can be reduced to $\Real^2$ (which is standard) via a suitable extension operator.
\end{proof}

\subsection{Assumptions on $\fraka$}
\label{sub:Assumptions_on_a}

We make the following assumptions on the nonlinear function $\fraka$. However, some results in this work hold under less restrictive conditions. When possible, we explicitly state the assumptions on $\fraka$ required for a specific result.

\begin{enumerate}[label=(A.\arabic*)]
\item \label{A1} $\fraka:\Omega\times \Real \rightarrow \Real$ is a Carath\'eodory function of class $C^2$ with respect to the second variable and $\fraka(\cdot,0)\in L^{r}(\Omega)$ for $r > 1/s$.

\item \label{A2} $\tfrac{\partial \fraka}{\partial u}(x,u)\geq 0$ for almost every $x\in\Omega$ and for all $u \in \Real$.

\item \label{A3} For all $\mathfrak{m}>0$, there exists a constant $C_{\mathfrak{m}} > 0$ such that 
\begin{equation*}
\sum_{i=1}^{2}\left|\frac{\partial^{i} \fraka}{\partial u^{i} }(x,u)\right|\leq C_{\mathfrak{m}},
\qquad
\left|\frac{\partial^{2} \fraka}{\partial u^{2} }(x,v)- \frac{\partial^{2} \fraka}{\partial u^{2} }(x,w)\right|\leq C_{\mathfrak{m}} |v-w|
\end{equation*}
for almost every $x\in \Omega$ and for all $u,v,w \in [-\mathfrak{m},\mathfrak{m}]$.
\end{enumerate}

\section{Fractional PDEs with measure-valued right-hand sides}
\label{sec:Fractional_PDEs_with_measures}

In this section, we analyze the following boundary value problem involving the spectral fractional Laplacian and a measure-valued right-hand side:
\begin{equation}
\label{eq:fractional_delta}
  \Laps \mathfrak{u} + c \mathfrak{u} = \mu \text{ in } \Omega.
\end{equation}
Here, the coefficient $c \in L^{\infty}(\Omega)$ satisfies $c \geq 0$ almost everywhere in $\Omega$, and the right-hand side $\mu \in \calM(\Omega)$. To introduce an appropriate formulation for this problem, we follow \cite{Fractional_Delta}, choose $\theta \in (1-s,s)$, and define the bilinear form
\begin{equation}
 \label{eq:mathcalA1}
  \calA : \polH^{s-\theta}(\Omega) \times \polH^{s+\theta}(\Omega) \to \Real,
  \qquad
  (v,w) \mapsto \calA(v,w) \coloneqq \sum_{k=1}^\infty \lambda_k^s v_k w_k,
\end{equation}
where
\begin{equation}
\label{eq:v_and_w}
  v = \sum_{k=1}^\infty v_k \varphi_k, 
  \qquad 
  w = \sum_{k=1}^\infty w_k \varphi_k.
\end{equation}
It is clear that the parameters $s$ and $\theta$ satisfy the following important inequalities:
\begin{equation}
\label{eq:s_and_theta}
 0 < s - \theta < 2s -1 < 1,
 \qquad
 1 < s + \theta < 2s < 2.
\end{equation}

With the bilinear form $\mathcal{A}$ at hand, we propose the following formulation for problem \cref{eq:fractional_delta}: Find $\mathfrak{u} \in \polH^{s-\theta}(\Omega)$ such that
\begin{equation}
\label{eq:WeakFormulation}
\calA(\mathfrak{u},v) + \int_{\Omega} c \mathfrak{u} v \mathrm{d}x= \langle \mu, v \rangle \qquad \forall v \in \polH^{s+\theta}(\Omega).
\end{equation}

The following comments are now in order. First, $\mathcal{A}$ is bounded and satisfies the inf-sup conditions stated in \cite[Theorem 3.2]{Fractional_Delta}. Second, $s \pm \theta > $ guarantees that $\mathbb{H}^{s\pm\theta}(\Omega) \hookrightarrow L^2(\Omega)$. As a result, the second term on the left-hand side of \cref{eq:WeakFormulation} is well-defined. Third, since $1< s + \theta < 2s$, we have $\polH^{s+\theta}(\Omega) \hookrightarrow C_0(\bar \Omega)$ (see \cref{pro:characterization}). Consequently, $\mu$ defines a bounded linear functional on $\mathbb{H}^{s+\theta}(\Omega)$, and the right-hand side of equation \cref{eq:WeakFormulation} is well-defined.

We now prove the well-posedness of problem \cref{eq:WeakFormulation}.

\begin{theorem}[well-posedness]
\label{thm:BNB}
For every $\mu \in \calM(\Omega)$, there is a unique solution $\mathfrak{u} \in \polH^{s-\theta}(\Omega)$ to problem \cref{eq:WeakFormulation}. In addition, $\mathfrak{u}$ satisfies the following stability bound
  \begin{equation}
  \label{eq:continuous_stability_measure}
    \| \mathfrak{u} \|_{\polH^{s-\theta}(\Omega)} \lesssim \| \mu \|_{\calM(\Omega)},
  \end{equation}
  where the implicit constant depends only on $s$, $\theta$, and $\Omega$.
\end{theorem}
\begin{proof}
If $c \equiv 0$, a proof of the well-posedness of problem \cref{eq:WeakFormulation} is given in \cite[Theorem 3.2]{Fractional_Delta}. In the case where $ 0 \leq c \in L^{\infty}(\Omega)$, we provide a proof based on the method of continuity; see, for example, \cite[Theorem 5.2]{MR1814364}. We divide the proof into several steps.

\emph{Step 1}. \emph{The map $L_0$.} We define the linear map $L_0 : \mathbb{H}^{s-\theta}(\Omega) \rightarrow \mathbb{H}^{-s-\theta}(\Omega)$ as
\[
\tensor[_{-s-\theta}]{\langle L_0 v, w \rangle}{_{s+\theta}} \coloneqq \mathcal{A}(v,w) 
\qquad 
\forall v \in \mathbb{H}^{s-\theta}(\Omega), \forall w \in \mathbb{H}^{s+\theta}(\Omega).
\]
The map $L_0$ is bounded: For every $v \in \mathbb{H}^{s-\theta}(\Omega)$, we have $\| L_0 v \|_{\mathbb{H}^{-s-\theta}(\Omega)} \leq \| v \|_{\mathbb{H}^{s-\theta}(\Omega)}$.

\emph{Step 2}. \emph{The map $L_1$.} We define the linear map $L_1 : \mathbb{H}^{s-\theta}(\Omega) \rightarrow \mathbb{H}^{-s-\theta}(\Omega)$ as
\[
\tensor[_{-s-\theta}]{\langle L_1 v, w \rangle}{_{s+\theta}} \coloneqq \mathcal{A}(v,w) + \int_{\Omega} c v w \mathrm{d}x 
\qquad 
\forall v \in \mathbb{H}^{s-\theta}(\Omega), \forall w \in \mathbb{H}^{s+\theta}(\Omega).
\]
$L_1$ is bounded: $\| L_1 v \|_{\mathbb{H}^{-s-\theta}(\Omega)} \leq (1 + C \| c \|_{L^{\infty}(\Omega)}) \| v \|_{\mathbb{H}^{s-\theta}(\Omega)}$ for every $v \in \mathbb{H}^{s-\theta}(\Omega)$. To obtain this bound, we used the continuity of the embeddings $\mathbb{H}^{s\pm\theta}(\Omega) \hookrightarrow L^2(\Omega)$, as well as the fact that $c \in L^{\infty}(\Omega)$. $C >0$ is a constant that depends on $s$, $\theta$, and $\Omega$.

\emph{Step 3}. \emph{The a priori estimate \cref{eq:continuous_stability_measure}}. For $t \in [0,1]$, we introduce the map $\mathcal{L}_t$ as
\[
 \mathcal{L}_t: \mathbb{H}^{s-\theta}(\Omega) \rightarrow \mathbb{H}^{-s-\theta}(\Omega),
 \qquad
 \mathcal{L}_t \coloneqq (1-t) L_0 + t L_1.
\]
We immediately note that $\calL_t$ is a homotopy between $L_0$ and $L_1$ in the space of bounded linear operators
from $\mathbb{H}^{s-\theta}(\Omega)$ to $\mathbb{H}^{-s-\theta}(\Omega)$. We now consider, for $t \in [0,1]$, the family of problems:
\begin{equation}
\label{eq:problem_t}
\mathfrak{u}_t \in \mathbb{H}^{s-\theta}(\Omega):
\quad
\tensor[_{-s-\theta}]{\langle \mathcal{L}_t \mathfrak{u}_t, v \rangle}{_{s+\theta}} =
 \langle \mu, v \rangle 
 \quad \forall v \in \polH^{s+\theta}(\Omega).
\end{equation}
The solvability of problem \cref{eq:problem_t} is therefore equivalent to the invertibility of the map $\mathcal{L}_t$. Let $\mathfrak{u}_t \in \mathbb{H}^{s-\theta}(\Omega)$ be a solution of problem \cref{eq:problem_t}. In the following, we prove that $\mathfrak{u}_t$ satisfies the bound
\begin{equation}
\label{eq:bound_t}
 \|  \mathfrak{u}_t \|_{\polH^{s-\theta}(\Omega)} \lesssim \| \mu \|_{\mathbb{H}^{-s-\theta}(\Omega)},
\end{equation}
where the hidden constant is independent of $t$. Note that the estimate above is equivalent to $\|  \mathfrak{u}_t \|_{\polH^{s-\theta}(\Omega)} \lesssim \|  \mathcal{L}_t \mathfrak{u}_t \|_{\polH^{-s-\theta}(\Omega)}$.

To derive \cref{eq:bound_t}, we proceed as follows. Let $\calL_t^\star$ denote the formal adjoint of $\calL_t$, which actually coincides with $\calL_t$. Given $\varphi \in C_0^\infty(\Omega)$, let $w = w(\varphi) \in \polH^s(\Omega)$ denote the solution to
\begin{equation}
\label{eq:w_solves_a_problem}
  \tensor[_{-s}]{\langle\calL_t^\star w, v \rangle}{_{s}} = \int_\Omega \varphi v \diff x \quad \forall v \in \polH^s(\Omega).
\end{equation}
The existence and uniqueness of such a  $w$ are guaranteed by the Lax-Milgram lemma, which also provides the estimate
\begin{equation}
\label{eq:ClaimForL2Estimate}
  \| w\|_{L^2(\Omega)} \lesssim \| w \|_{\polH^s(\Omega)} \lesssim \| \varphi \|_{\polH^{-s}(\Omega)} \lesssim \| \varphi \|_{L^2(\Omega)}.
\end{equation}
To obtain \cref{eq:ClaimForL2Estimate}, we have also used the Sobolev embeddings $\polH^s(\Omega) \hookrightarrow L^2(\Omega)$ and $L^2(\Omega) \hookrightarrow  \polH^{-s}(\Omega)$. We must also note that this $w$ solves the problem
\[
  \Laps w = \varphi - t c w \quad \text{in } \Omega,
  \qquad
  \varphi - t c w \in \Ldeux.
\]
It follows directly from the definitions of $\Laps$ and the spaces $\polH^r(\Omega)$ that $w \in \polH^{2s}(\Omega) \hookrightarrow \polH^{s+\theta}(\Omega)$; the latter holds because $1< s + \theta < 2s$. In addition, we have
\[
  \| w \|_{\polH^{s+\theta}(\Omega)} \lesssim \| w \|_{\polH^{2s}(\Omega)} \lesssim \| \varphi - t c w \|_{\Ldeux} \leq \| \varphi \|_\Ldeux + \| c \|_{L^\infty(\Omega)} \| w \|_{L^2(\Omega)}.
\]
We now use bound \cref{eq:ClaimForL2Estimate} to conclude that
\begin{equation}
\label{eq:HsptEstimateInLieuOfUniqueness}
  \| w \|_{\polH^{s+\theta}(\Omega)} \lesssim \| \varphi \|_\Ldeux,
\end{equation}
where the hidden constant is independent of $t$ but depends on $s$, $\theta$, $\| c\|_{\Linf}$, and $\Omega$.

The next step toward proving  \cref{eq:bound_t} is to realize that
$\mathfrak{u}_t$ can be viewed as the solution to the following problem:
\begin{equation}
\label{eq:problem_t_can_be_seen}
\mathfrak{u}_t \in \mathbb{H}^{s-\theta}(\Omega):
\quad
 \tensor[_{-s-\theta}]{\langle L_0 \mathfrak{u}_t, v \rangle}{_{s+\theta}} =
 \langle \mu, v \rangle  - t \int_{\Omega} c \mathfrak{u}_t v \mathrm{d}x
\quad \forall v \in \polH^{s+\theta}(\Omega).
\end{equation}
Applying \cite[Theorem 3.2]{Fractional_Delta}, we deduce the bound 
\begin{equation}
 \label{eq:bound_Garding}
 \|  \mathfrak{u}_t \|_{\polH^{s-\theta}(\Omega)} \lesssim \|  \mu \|_{\polH^{-s-\theta}(\Omega)} + C \| c \|_{L^{\infty}(\Omega)} \|  \mathfrak{u}_t \|_{L^2(\Omega)}
 \quad
 \forall t \in [0,1],
\end{equation}
where $C>0$ depends only on $s$, $\theta$, and $\Omega$. Thus, we now focus on controlling the term $\| \mathfrak{u}_t \|_{L^2(\Omega)}$ on the right-hand side of \cref{eq:bound_Garding}.
To do this, let $\varphi \in C_0^\infty(\Omega)$ be arbitrary. We use problem \cref{eq:w_solves_a_problem}, the regularity result $w \in \polH^{2s}(\Omega) \hookrightarrow \polH^{s+\theta}(\Omega)$, the definition of the formal adjoint $\calL_t^\star$ of $\calL_t$, and the fact that $\mathfrak{u}_t$ solves \cref{eq:problem_t} to obtain
\begin{align*}
  \int_\Omega \fraku_t \varphi \diff x &= \int_\Omega \fraku_t \calL_t^\star w(\varphi) \diff x =
  \tensor[_{-s-\theta}]{\langle \calL_t \fraku_t, w(\varphi) \rangle}{_{s+\theta}} = 
  \tensor[_{-s-\theta}]{\langle \mu, w(\varphi) \rangle}{_{s+\theta}}
  \\
  &\leq \| \mu \|_{\polH^{-s-\theta}(\Omega)} \| w(\varphi) \|_{\polH^{s+\theta}(\Omega)}
  \lesssim \| \mu \|_{\polH^{-s-\theta}(\Omega)} \| \varphi \|_\Ldeux,
\end{align*}
where, in the last step, we invoked \cref{eq:HsptEstimateInLieuOfUniqueness}. We can therefore conclude that
\[
  \| \fraku_t \|_\Ldeux = \sup_{0 \neq \varphi \in C_0^\infty(\Omega)} \frac1{\| \varphi \|_\Ldeux } \int_\Omega \fraku_t \varphi \diff x \lesssim \| \mu \|_{\polH^{-s-\theta}(\Omega)}.
\]
Replacing this bound in  \cref{eq:bound_Garding} implies \cref{eq:bound_t}.

\emph{Step 3}. \emph{The method of continuity}. Since $L_0$ and $L_1$ are linear and bounded operators from $\mathbb{H}^{s-\theta}(\Omega)$ into $\mathbb{H}^{-s-\theta}(\Omega)$, and the bound \cref{eq:bound_t} holds, we can apply \cite[Theorem 5.2]{MR1814364} to conclude that $L_0$ maps $\mathbb{H}^{s-\theta}(\Omega)$ onto $\mathbb{H}^{-s-\theta}(\Omega)$ if and only if $L_1$ maps $\mathbb{H}^{s-\theta}(\Omega)$ onto $\mathbb{H}^{-s-\theta}(\Omega)$. As a result, $L_1$ maps $\mathbb{H}^{s-\theta}(\Omega)$ onto $\mathbb{H}^{-s-\theta}(\Omega)$; that is, problem \cref{eq:WeakFormulation} has a solution. We now use the linearity of problem \cref{eq:WeakFormulation} and the stability bound \cref{eq:bound_t} to conclude that problem \cref{eq:WeakFormulation} has a unique solution. This completes the proof.
\end{proof} 

\section{Fractional, semilinear PDEs}
\label{sec:Fractional_semilinear_PDEs}

Let $\mathsf{f} \in L^r(\Omega)$ with $r > 1/s$, and let $\frakb:\Omega\times \Real \rightarrow \Real$ be a Carath\'eodory function that is monotone increasing in its second argument. Assume that, for every $\mathfrak{m}>0$, there exists $\psi_{\mathfrak{m}} \in L^r(\Omega)$ such that
\begin{equation}
\label{eq:a_bounded_for_psi_m}
 |\frakb(x,\zeta)| \leq \psi_{\mathfrak{m}}(x)
\end{equation}
for almost every $x \in \Omega$ and $\zeta \in [-\mathfrak{m},\mathfrak{m}]$. In this context, we present the following formulation for a slight variant of the fractional, semilinear PDE \cref{eq:state_equation_strong}:
\begin{equation}
 \label{eq:Fractional_semilinear_PDE_weak}
 \mathsf{u} \in \mathbb{H}^s(\Omega):
 \quad
 \mathcal{B}(\mathsf{u},v) + \int_{\Omega} \frakb(\cdot, \mathsf{u}) v \mathrm{d}x = \int_{\Omega} \mathsf{f} v \mathrm{d}x
 \quad
 \forall v \in \mathbb{H}^s(\Omega).
 \end{equation}
Here, $\mathcal{B}$ is defined as
\begin{equation}
 \label{eq:mathcalB}
  \mathcal{B} : \polH^{s}(\Omega) \times \polH^{s}(\Omega) \to \Real,
  \qquad
  (v,w) \mapsto \mathcal{B}(v,w) \coloneqq \sum_{k=1}^\infty \lambda_k^s v_k w_k.
\end{equation}
It is clear that $\mathcal{B}$ is bilinear and continuous on $\polH^s(\Omega) \times \polH^s(\Omega)$.

We present the following well-posedness result.

\begin{proposition}[well-posedness]
\label{pro:well-posedness_fractional_semilinear_PDE}
Assume that $\mathsf{f} \in L^r(\Omega) \cap \polH^{-s}(\Omega)$ with $r>1/s$ and that \cref{eq:a_bounded_for_psi_m} holds. In this setting, problem \cref{eq:Fractional_semilinear_PDE_weak} has a unique solution which, additionally, satisfies $\mathsf{u} \in \mathbb{H}^s(\Omega) \cap L^{\infty}(\Omega)$. In addition, we have the bound
 \begin{equation}
  \label{eq:continuous_stability_semilinear_PDE}
  \| \mathsf{u} \|_{\mathbb{H}^s(\Omega)} + \| \mathsf{u} \|_{L^{\infty}(\Omega)} \lesssim \| \mathsf{f} - \frakb(\cdot,0) \|_{L^r(\Omega)}.
 \end{equation}
\end{proposition}
\begin{proof}
Assume, for the moment, that \cref{eq:a_bounded_for_psi_m} holds for all $\zeta \in \Real$ and that $\frakb(\cdot,0) = 0$. Define the map
\[
 \mathfrak{A}: \polH^s(\Omega) \rightarrow \polH^{-s}(\Omega),
 \qquad
 \langle \mathfrak{A}v,w \rangle \coloneqq \mathcal{B}(v,w) + \int_{\Omega} \frakb(\cdot,v) w \mathrm{d}x
\]
for every $v,w \in \polH^s(\Omega)$. The bilinear form $\mathcal{B}$ is continuous and coercive in $\polH^s(\Omega) \times \polH^s(\Omega)$. In addition, the function $\frakb$ satisfies $|\frakb(x,\zeta)| \leq \psi(x)$ for almost every $x \in \Omega$ and for all $\zeta \in \Real$ and it is continuous and monotone increasing with respect to its second argument. Thus, it can be proved that $\mathfrak{A}$ is well-defined, strongly monotone, coercive, and hemicontinuous. The existence and uniqueness of a solution $\mathsf{u}  \in \mathbb{H}^s(\Omega)$ follow from the Browder--Minty theorem \cite[Theorem 2.18]{MR3014456}. A stability bound for $\mathsf{u}$ in $\polH^s(\Omega)$ is obtained by setting $v = \mathsf{u}$ in \cref{eq:Fractional_semilinear_PDE_weak} and using the fact that $\frakb(\cdot,0) = 0$. The fact that $\mathsf{u} \in L^{\infty}(\Omega)$ and the stability bound $\| \mathsf{u} \|_{L^{\infty}(\Omega)} \lesssim
\| \mathsf{f} \|_{L^r(\Omega)}$
can be found in \cite[Theorem 2.9]{MR3745164}. The final step is to relax the global assumption on $\frakb$, namely,
$|\frakb(x,\zeta)| \leq \psi(x)$ for almost every $x \in \Omega$ and for all $\zeta \in \Real$, and require only the local assumption \cref{eq:a_bounded_for_psi_m}. This can be done using the arguments presented in the proof of \cite[Theorem 4.7]{MR2583281}. We remove the assumption $\frakb(\cdot , 0) = 0$ by replacing $\frakb(\cdot , \zeta)$ with $\frakb(\cdot , \zeta) - \frakb(\cdot , 0)$. This concludes the proof.
\end{proof}

We now present the following regularity result.

\begin{theorem}[regularity]
\label{thm:regularity_fractional_semilinear_PDE}
Assume that $\mathsf{f}$, $\frakb(\cdot,0) \in L^2(\Omega)$, and that the function $\frakb$ is locally Lipschitz in its second argument; that is, for every $\mathfrak{m}>0$, there is $\frakC_\frakm>0$ such that
 \begin{equation}
 \label{eq:a_is_Lipschitz}
  |\frakb(x,\zeta_1) - \frakb(x,\zeta_2)| \leq \mathfrak{C}_{\mathfrak{m}} | \zeta_1 - \zeta_2| \qquad \mae~ x\in \Omega, \quad \forall \zeta_1, \zeta_2 \in [\frakm, \frakm].
 \end{equation}
 Then, problem \cref{eq:Fractional_semilinear_PDE_weak} has a unique solution $\mathsf{u}$ which, in addition, belongs to $\mathbb{H}^{2s}(\Omega) \cap C(\bar \Omega)$ and satisfies
 \begin{equation}
  \label{eq:regularity_semilinear_PDE}
  \| \mathsf{u} \|_{\mathbb{H}^{2s}(\Omega)} + \| \mathsf{u} \|_{C(\bar \Omega)} \lesssim \| \mathsf{f} - \frakb(\cdot,0) \|_{L^2(\Omega)}.
 \end{equation}
\end{theorem}
\begin{proof}
We may use $\frakb(\cdot,0) \in L^2(\Omega)$ and \cref{eq:a_is_Lipschitz} to see that, for every $\zeta \in [-\frakm, \frakm]$,
\[
  |\frakb(\cdot,\zeta)| \leq \left| \frakb(\cdot,\zeta) - \frakb(\cdot,0) \right| + \left| \frakb(\cdot,0) \right| \leq \mathfrak{C}_\frakm \frakm + |\frakb(\cdot,0)| \eqqcolon \psi_\frakm(\cdot) \in \Ldeux.
\]
Then, since $2 > \tfrac1s$, \cref{eq:a_bounded_for_psi_m} holds. Consequently, the results of \cref{pro:well-posedness_fractional_semilinear_PDE} apply (recall that $\mathsf{f} \in L^2(\Omega)$), guaranteeing the existence and uniqueness of $\mathsf{u} \in \polH^s(\Omega) \cap L^\infty(\Omega)$ that solves \cref{eq:Fractional_semilinear_PDE_weak}. We now define $\mathsf{g} \coloneqq \mathsf{f} - \frakb(\cdot,\mathsf{u})$. Note that $\mathsf{g} \in L^2(\Omega)$. In fact,
 \begin{multline}
  \| \mathsf{g} \|_{L^2(\Omega)} \leq 
  \| \mathsf{f} - \frakb(\cdot,0) \|_{L^2(\Omega)} + \| \frakb(\cdot,0) - \frakb(\cdot,\mathsf{u})\|_{L^2(\Omega)}
  \\
  \leq 
  \| \mathsf{f} - \frakb(\cdot,0) \|_{L^2(\Omega)} + C_{\mathsf{m}} \| \mathsf{u} \|_{L^2(\Omega)},
 \end{multline}
where we have used \cref{eq:a_is_Lipschitz} with $\mathfrak{m} = \mathsf{m} \coloneqq  \| \mathsf{u} \|_{L^{\infty}(\Omega)}$. It thus follows directly from the definition of $(-\Delta)^s$ and the spaces $\polH^r(\Omega)$ that $\mathsf{u} \in \polH^{2s}(\Omega)$, together with the bound
\begin{equation}
\label{eq:u_H2s}
 \| \mathsf{u} \|_{\polH^{2s}(\Omega)} \lesssim \| \mathsf{f} - \frakb(\cdot,0) \|_{L^2(\Omega)} + C_{\mathsf{m}} \| \mathsf{u} \|_{L^2(\Omega)} \lesssim \| \mathsf{f} - \frakb(\cdot,0) \|_{L^2(\Omega)},
\end{equation}
where we have used the stability bound \cref{eq:continuous_stability_semilinear_PDE}. We now invoke \cref{pro:characterization} and the fact that $s > \frac{1}{2}$ to deduce that $\mathsf{u} \in C(\bar \Omega)$. Since the embedding $\mathbb{H}^{2s}(\Omega) \hookrightarrow C(\bar \Omega)$  is continuous, the desired bound follows from \cref{eq:u_H2s}. This concludes the proof.
\end{proof}

We conclude this section with the following Lipschitz property.

\begin{theorem}[Lipschitz property]
\label{thm:Lipschitz_property_state_equation}
Assume that $\frakb$ is of class $C^1$ with respect to the second variable, that $\frakb(\cdot,0) \in L^2(\Omega)$, and that \cref{eq:a_bounded_for_psi_m} and \Cref{A2} hold. Additionally, suppose that
\begin{equation}
\label{eq:bounded_on_the_der_b}
  \left| \frac{ \partial \frakb }{ \partial u } (x,\zeta) \right| \leq \mathfrak{C}_{\mathfrak{m}} \qquad \mae \ x \in \Omega, \quad \zeta \in [-\frakm, \frakm].
\end{equation}
Let $\mathsf{f}_1, \mathsf{f}_2 \in L^2(\Omega)$, and let $\mathsf{u}_i$ be the solution to \cref{eq:Fractional_semilinear_PDE_weak} with $\mathsf{f}$ replaced by $\mathsf{f}_i$, where $i \in \{1,2\}$. Then,
 \begin{equation}
  \label{eq:Lipschitz_property}
  \| \mathsf{u}_1 - \mathsf{u}_2 \|_{\mathbb{H}^{2s}(\Omega)} + \|\mathsf{u}_1 - \mathsf{u}_2 \|_{C(\bar \Omega)} \lesssim \| \mathsf{f}_1 - \mathsf{f}_2 \|_{L^2(\Omega)}.
 \end{equation}
\end{theorem}
\begin{proof}
The proof follows from a straightforward adaptation of the arguments in \cite[Theorem 4.16]{MR2583281}. Define $\mathtt{u} \coloneqq \mathsf{u}_2 - \mathsf{u}_1$, and note that $\mathtt{u}$ solves the problem
\begin{equation}
\label{eq:LinearizationForLipschitz}
  \mathcal{B}(\mathtt{u},v) + \int_{\Omega} \mathtt{c} \mathtt{u} v \mathrm{d}x = \int_{\Omega} \mathtt{f} v \mathrm{d}x
  \quad
  \forall v \in \mathbb{H}^s(\Omega),
\end{equation}
where
$
  \mathtt{f}  \coloneqq \mathsf{f}_2 - \mathsf{f}_1,
$
and
\[
  \mathtt{c}(x) \coloneqq \int_0^1 \frac{\partial \frakb}{\partial u} (x,\mathsf{u}_1(x) + \tau(\mathsf{u}_2(x) - \mathsf{u}_1(x))) \mathrm{d}\tau .
\]
We now notice that, for every $\tau \in [0,1]$, the bound \cref{eq:regularity_semilinear_PDE} implies
\begin{align*}
  \| \mathsf{u}_1  + \tau( \mathsf{u}_2 - \mathsf{u}_1 ) \|_{L^\infty(\Omega)} &\leq (1-\tau) \| \mathsf{u}_1 \|_{L^\infty(\Omega)} + \tau \|\mathsf{u}_2 \|_{L^\infty(\Omega)} \\
  &\leq C\left[ (1-\tau) \| \mathsf{f}_1 - \frakb(\cdot,0) \|_{L^2(\Omega)}  + \tau \| \mathsf{f}_2 - \frakb(\cdot,0) \|_{L^2(\Omega)} \right]\\
  &\leq C \max\left\{ \| \mathsf{f}_1 - \frakb(\cdot,0) \|_{L^2(\Omega)}, \| \mathsf{f}_2 - \frakb(\cdot,0) \|_{L^2(\Omega)} \right\} \eqqcolon \mathsf{m}.
\end{align*}
In addition, by \Cref{A2}, we have that $\mathtt{c} \geq 0$ a.e.~in $\Omega$. Finally, the boundedness assumption \cref{eq:bounded_on_the_der_b} on the derivative of $\frakb$ gives
$
0  \leq \mathtt{c}(x) \leq \frakC_\mathsf{m}
$
for a.e.~$x \in \Omega$, and so we may conclude that $\mathtt{c}\in L^{\infty}(\Omega)$. Problem \cref{eq:LinearizationForLipschitz} then fits into the assumptions of \cref{thm:regularity_fractional_semilinear_PDE}. Indeed, we may define
\[
  \frakc(x,\zeta) \coloneqq \mathtt{c}(x) \zeta
\]
to trivially fit this setup. The desired bound is then simply a restatement of \cref{eq:regularity_semilinear_PDE}.
\end{proof}

\section{The optimal control problem}
\label{sec:Control_problem}
We now have all the elements needed to provide a precise description of the pointwise tracking optimal control problem introduced in \Cref{sec:Introduction}. 

Given $f \in L^2(\Omega)$ and a Carath\'eodory function $\fraka:\Omega\times \Real \rightarrow \Real$ that is monotone increasing in the second argument and satisfies $\fraka(\cdot,0) \in L^2(\Omega)$ and \cref{eq:a_is_Lipschitz}, the optimal control problem of interest is: Find
\begin{equation}
 \label{eq:min}
\min \left\{ J(u,q): (u,q) \in \mathbb{H}^s(\Omega) \cap C(\bar \Omega) \times \mathbb{Q}_{ad} \right\},
\end{equation}
subject to the fractional, semilinear, elliptic PDE \cref{eq:state_equation_strong}, understood as
\begin{equation}
 \label{eq:state_equation_weak}
 u \in \mathbb{H}^s(\Omega):
 \quad
 \mathcal{B}(u,v) + \int_{\Omega} \fraka(\cdot, u) v \mathrm{d}x = \int_{\Omega} (f + q) v \mathrm{d}x
 \quad
 \forall v \in \mathbb{H}^s(\Omega).
 \end{equation}
We recall that the cost functional $J$ is defined in \cref{eq:cost_functional}, and the set $\mathbb{Q}_{ad} \subset L^2(\Omega)$ is defined in \cref{eq:DefOfQ}. Given $q \in \polQ_{ad}$, since $f + q \in L^2(\Omega)$, an immediate application of \cref{thm:regularity_fractional_semilinear_PDE} shows that \cref{eq:state_equation_weak} has a unique solution $u \in \polH^{2s}(\Omega) \cap C(\bar \Omega)$. Thus, point evaluations of the state $u$ are well-defined, and consequently, so is the functional $J$.

As it is customary in PDE-constrained optimization, we introduce the \emph{control-to-state map} as follows:
\[
 \mathcal{S}: L^2(\Omega) \rightarrow \mathbb{H}^s(\Omega) 
\]
which maps $q \in L^2(\Omega)$ to the unique $u \in \mathbb{H}^s(\Omega)$ that solves the state equation \cref{eq:state_equation_weak} (see \cref{pro:well-posedness_fractional_semilinear_PDE} and \cref{thm:regularity_fractional_semilinear_PDE}). Since, as we have just discussed, the range of $\calS$ is contained in $\polH^{2s}(\Omega)  \cap C(\bar\Omega)$,
we can define the reduced cost functional
\begin{equation}
 \label{eq:reduced_cost_functional}
  j: \mathbb{Q}_{ad} \rightarrow \Real,
  \qquad
  q \mapsto j(q) \coloneqq J(\mathcal{S}q,q) = \frac{1}{2} \sum_{\vertex \in \mathcal{D}} | \calS q(\vertex) - u_\vertex |^2 + \frac{\alpha}{2} \| q \|^2_{L^2(\Omega)}.
\end{equation}

\subsection{Existence of solutions}
\label{subsec:Existence}

The following basic result states that the optimal control problem \cref{eq:min,eq:state_equation_weak} has at least one global optimal solution $(\bar u, \bar q) \in \mathbb{H}^s(\Omega) \cap C(\bar \Omega) \times \mathbb{Q}_{ad}$. The proof follows standard arguments similar to those used in the proof of \cite[Theorem 4.15]{MR2583281}. Below, we provide a brief proof that emphasizes the role of the fractional operator $(-\Delta)^s$ and the inclusion of point evaluations of the state in the cost functional $J$ defined in \cref{eq:cost_functional}.

\begin{theorem}[existence]
The pointwise tracking optimal control problem defined by \cref{eq:min,eq:state_equation_weak} has at least one global optimal solution $(\bar u, \bar q) \in \mathbb{H}^s(\Omega) \cap C(\bar \Omega) \times \mathbb{Q}_{ad}$.
\end{theorem}
\begin{proof}
Define
\[
  \mathfrak{j} \coloneqq \inf \left\{ J(u,q): (u,q) \in \polH^s(\Omega) \cap C(\bar \Omega) \times \mathbb{Q}_{ad}, \, u = \mathcal{S}q \right\} \geq 0.
\]
Let $\{ (u_k,q_k) \}_{k \in \Natural}$ be an infimizing sequence, i.e., $q_k \in \mathbb{Q}_{ad}$ and $u_k = \mathcal{S} q_k \in \polH^s(\Omega) \cap C(\bar \Omega)$ are such that $J(u_k,q_k) \rightarrow \mathfrak{j}$ as $k \uparrow \infty$. Since $\{ q_k \}_{k \in \Natural} \subset \mathbb{Q}_{ad}$ and $\mathbb{Q}_{ad}$ is nonempty, closed, bounded, and convex in $L^2(\Omega)$, we conclude that there exists a, nonrelabeled, subsequence $\{ q_k \}_{k \in \Natural}$ such that $q_k \rightharpoonup \bar{q}$ in $L^2(\Omega)$ as $k \uparrow \infty$ and $\bar q \in \mathbb{Q}_{ad}$.

As the next step, we apply bound \cref{eq:regularity_semilinear_PDE} from \cref{thm:regularity_fractional_semilinear_PDE} and use the fact that $\mathbb{Q}_{ad}$ is bounded in $L^2(\Omega)$, to conclude that there exists $\mathfrak{m} > 0$ such that $\| u_k \|_{C(\bar \Omega)} \leq \mathfrak{m}$ for all $k \in \Natural$. Having established this uniform bound, we use the Lipschitz property \cref{eq:a_is_Lipschitz} to obtain that $\{ \fraka(\cdot,u_k) \}_{k \in \Natural}$ is bounded in $L^2(\Omega)$. In fact, we have
\[
 \| \fraka(\cdot,u_k) \|_{L^2(\Omega)} \leq  \| \fraka(\cdot,0) \|_{L^2(\Omega)} + \mathfrak{C}_{\mathfrak{m}} \| u_k \|_{L^2(\Omega)}
 \quad
 \forall k \in \Natural.
\]
Recall that, by assumption, $\fraka(\cdot,0) \in L^2(\Omega)$. Next, we rewrite the problem
\begin{equation*}
u_k \in \polH^s(\Omega):
\qquad
\mathcal{B} (u_k,v) + \int_{\Omega} \fraka(\cdot,u_k)v \mathrm{d}x = \int_{\Omega} (f + q_k)v \mathrm{d}x
\qquad
\forall v \in \mathbb{H}^s(\Omega)
\end{equation*}
as follows:
\begin{equation}
\label{eq:u_k}
u_k \in \polH^s(\Omega):
\qquad
\mathcal{B} (u_k,v) = \int_{\Omega} \left[ f + q_k - \fraka(\cdot,u_k) \right] v \mathrm{d}x
\qquad
\forall v \in \mathbb{H}^s(\Omega).
\end{equation}
We may then set $\frakb \equiv 0$ and $\mathsf{f} \coloneqq f + q_k - \fraka(\cdot,u_k)$ to fit \cref{eq:u_k} into \cref{eq:Fractional_semilinear_PDE_weak} trivially. Since the assumptions of \Cref{thm:regularity_fractional_semilinear_PDE} hold trivially in this case, we have
\[
  \| u_k \|_{\polH^{2s}(\Omega)} + \| u_k \|_{C(\bar\Omega)} \lesssim \| f + q_k - \fraka(\cdot,u_k) \|_{L^2(\Omega)} \lesssim 1,
   \qquad
 \forall k \in \Natural.
\]
Here, we used the fact that $\{q_k\}_{k \in \Natural}$ and $\{ \fraka(\cdot,u_k) \}_{k \in \Natural}$ are bounded in $L^2(\Omega)$. We may then extract a further nonrelabeled subsequence $\{ u_k \}_{k \in \Natural} \subset \polH^{2s}(\Omega)$ such that $u_k \rightharpoonup \bar{u}$ in $\polH^{2s}(\Omega)$ as $k \uparrow \infty$. Since the embedding $\polH^{2s}(\Omega) \hookrightarrow \polH^s(\Omega) \cap C(\bar\Omega)$ is compact, up to a further subsequence, we have
$
  u_k \to \bar{u}
$
in $\polH^s(\Omega) \cap C(\bar\Omega)$ as $k \uparrow \infty$. Such strong convergence is sufficient to take the limit in \cref{eq:u_k} and conclude that $\bar{u} = \calS(\barq)$.


%

Finally, we show that $(\bar u, \bar q)$ is optimal. The uniform convergence $u_k \to \bar{u}$ as $k \uparrow \infty$ implies that
\begin{equation}
 \frac{1}{2} \sum_{\vertex \in \mathcal{D}} | u_k(\vertex) - u_\vertex |^2 \rightarrow \frac{1}{2} \sum_{\vertex \in \mathcal{D}} | \bar{u}(\vertex) - u_\vertex |^2,
 \qquad
 k \uparrow \infty.
 \label{eq:convergence_functional_pointwise}
\end{equation}
This, together with the weak convergence $q_k \rightharpoonup \bar{q}$ as $k \uparrow \infty$ in $L^2(\Omega)$, allows us to deduce
\[
 \mathfrak{j} = \lim_{k \uparrow \infty} J(u_k,q_k)
 \geq \lim_{k \uparrow \infty} \frac12 \sum_{\vertex \in \mathcal{D}} | u_k(\vertex) - u_\vertex |^2
 +
 \liminf_{k \uparrow \infty} \frac{\alpha}{2} \| q_k \|^2_{L^2(\Omega)} \geq J(\bar u, \bar q),
\]
which shows that $J(\bar u, \bar q) = \mathfrak{j}$. This concludes the proof.
\end{proof}

\subsection{First order optimality conditions}

Throughout this section, we assume that \Cref{A1,A2,A3} hold. In addition, we assume that $f$ and $\fraka(\cdot,0)$ belong to $L^2(\Omega)$.

\subsubsection{Differentiability properties of the control-to-state map $\mathcal{S}$}

In the following result, we state differentiability properties of the control-to-state map $\mathcal{S}$, which will be useful for deriving both first and second order optimality conditions. Although similar results appear in \cite[Theorem 4.3]{MR4358465} and \cite[Lemma 5.3]{MR4055455}, we present a statement and a brief proof adapted to our needs.

\begin{theorem}[differentiability of $\mathcal{S}$]
\label{thm:differentiability}
If \Cref{A1,A2,A3} hold, then the control-to-state map $\mathcal{S}: L^2(\Omega) \rightarrow \mathbb{H}^s(\Omega)$ is of class $C^2$. In addition, if $q,w \in L^2(\Omega)$, then $\phi = \calS'(q)w \in \mathbb{H}^s(\Omega)$ corresponds to the unique solution to
\begin{equation}
\label{eq:first_derivative}
\phi \in \mathbb{H}^s(\Omega):
\quad
\mathcal{B}(\phi,v) + \int_{\Omega} \frac{\partial \fraka}{\partial u}(\cdot,u) \phi v \mathrm{d}x = \int_{\Omega} w v \mathrm{d}x
\quad
\forall v \in \mathbb{H}^{s}(\Omega),
\end{equation}
where $u = \mathcal{S}q$. If $q,w_1,w_2 \in L^2(\Omega)$, then $\psi = \calS''(q)(w_1,w_2) \in \mathbb{H}^s(\Omega)$ corresponds to the unique solution to the problem
\begin{equation}
\label{eq:second_derivative}
\psi \in \mathbb{H}^s(\Omega):
\quad
\mathcal{B}(\psi,v) + \int_{\Omega} \frac{\partial \fraka}{\partial u}(\cdot,u) \psi v \mathrm{d}x
=
-
\int_{\Omega} \frac{\partial^2 \fraka}{\partial u^2}(\cdot,u) \phi_{w_1}\phi_{w_2} v \mathrm{d}x
\end{equation}
for all $v \in \mathbb{H}^{s}(\Omega)$. Here, $u = \mathcal{S}q$ and $\phi_{w_i} = \mathcal{S}'(q) w_i$, with $i \in \{1,2\}$.
\end{theorem}
\begin{proof}
Define the map
\[
 F: L^2(\Omega) \times \polH^{2s}(\Omega) \rightarrow L^2(\Omega),
 \qquad
 (\kappa,v) \mapsto F(\kappa,v) \coloneqq
 \Laps v + \fraka(\cdot,v) - f - \kappa.
\]
Since $s > \frac12$, we have $\mathbb{H}^{2s}(\Omega) \hookrightarrow C(\bar \Omega)$ (see \cref{pro:characterization}). We may then set $\frakm =\| v\|_{C(\bar\Omega)}$ and use \Cref{A3} to obtain
\[
  \| \fraka (\cdot, v) \|_{L^2(\Omega)} \leq \| \fraka(\cdot, v) - \fraka(\cdot,0) \|_{L^2(\Omega)} + \| \fraka (\cdot, 0) \|_{L^2(\Omega)} \leq
  \mathfrak{m} C_{\mathfrak{m}} |\Omega|^{1/2} + \| \fraka (\cdot, 0) \|_{L^2(\Omega)}.
\]
As a result, $F$ is well-defined. It can also be deduced from \Cref{A1,A2,A3} that $F$ is of class $C^2$.

We now let $\bar{q} \in L^2(\Omega)$, and set $\bar{u} = \mathcal{S} \bar{q}$. It is immediate that $F(\bar{q},\bar{u}) = 0$. On the other hand, the derivative $\tfrac{\partial F}{\partial u} (\bar{q},\bar{u})$ satisfies
\[
 \frac{\partial F}{\partial u} (\bar{q},\bar{u}): \polH^{2s}(\Omega) \rightarrow L^2(\Omega),
 \qquad
 \frac{\partial F}{\partial u} (\bar{q},\bar{u}) w = \Laps w + \frac{\partial \fraka}{\partial u}(\cdot,\bar{u})w.
\]
It is clear that $\tfrac{\partial F}{\partial u} (\bar{q},\bar{u})$ is a linear and continuous map. In addition, the assumptions on $\fraka$ allow us to deduce that $\tfrac{\partial F}{\partial u} (\bar{q},\bar{u})$ is an isomorphism. Indeed, it suffices to set
\[
  \frakb(\cdot, w) = \frac{\partial \fraka}{\partial u}(\cdot,\bar{u}) w
\]
in \cref{eq:Fractional_semilinear_PDE_weak} and apply \cref{thm:regularity_fractional_semilinear_PDE}. In summary, we have all the ingredients to apply the implicit function theorem and conclude that the control-to-state map $\mathcal{S}$ is of class $C^2$. The fact that $\phi$ and $\psi$ solve \cref{eq:first_derivative} and \cref{eq:second_derivative}, respectively, follows from differentiating
the relation $F(\bar{q},\mathcal{S}\bar{q}) = 0$; see, for example, the proof of \cite[Theorem 4.24(ii)]{MR2583281} for details.
\end{proof}

\subsubsection{Local solutions}
In the absence of convexity, we work within the framework of local solutions in the sense of $L^2(\Omega)$.

\begin{definition}[local solution]
    We say that a control $\bar{q} \in \mathbb{Q}_{ad}$ is \emph{locally optimal} for \cref{eq:min,eq:state_equation_weak}, in the sense of $L^2(\Omega)$, if there exists $\varepsilon > 0$ such that for all $q \in \mathbb{Q}_{ad}$ that satisfy
    \[
      \| q -\bar{q} \|_{L^2(\Omega)} \leq \varepsilon,
    \]
    we have $j(\bar{q}) \leq j(q)$.
\end{definition}

To present the local optimality results, we analyze the differentiability properties of the reduced cost functional $j$.

\begin{proposition}[differentiability of $j$]
\label{pro:j_is_of_class_C2}
 If \Cref{A1,A2,A3} hold, then, the reduced cost functional $j: L^2(\Omega) \rightarrow \Real$ is of class $C^2$.
\end{proposition}
\begin{proof}
 This result follows directly from the chain rule and the results in \cref{thm:differentiability}.
\end{proof}

The following result is standard \cite[Lemma 4.18]{MR2583281}: If $\bar{q} \in \mathbb{Q}_{ad}$ is a locally optimal control for the pointwise tracking optimal control problem \cref{eq:min,eq:state_equation_weak}, then
\begin{equation}
\label{eq:first_order_basic}
 j'(\bar q)(q - \bar{q}) \geq 0
 \quad
 \forall q \in \mathbb{Q}_{ad}.
\end{equation}
Here, $j'(\bar{q})$ denotes the Gateaux derivative of $j$ at $\bar{q}$.

\subsubsection{The adjoint problem} As is customary in PDE-constrained optimization, we further examine \cref{eq:first_order_basic} by introducing the so-called adjoint equation: Find $p \in \mathbb{H}^{s - \theta}(\Omega)$ such that
\begin{equation}
\label{eq:adjoint_equation_weak}
\calA(p,v) + \int_{\Omega} \frac{\partial \fraka}{\partial u}(\cdot,u)  p v \mathrm{d}x
=
\sum_{\vertex \in \mathcal{D}} (u(\vertex) - u_\vertex) \langle \delta_\vertex, v \rangle \quad \forall v \in \polH^{s+\theta}(\Omega).
\end{equation}
Here, $u = \mathcal{S} q$ and, as in \Cref{sec:Fractional_PDEs_with_measures}, $\theta \in (1-s,s)$.

As shown in the following result, the adjoint problem is well-posed.

\begin{theorem}[well-posedness]
\label{thm:adjoint_well-posedness}
Given $q \in \polQ_{ad}$, let $u = \calS q$. Then,
 the adjoint problem \cref{eq:adjoint_equation_weak} is well-posed. In particular, the unique solution $p \in \mathbb{H}^{s - \theta}(\Omega)$ satisfies the following stability bound:
 \begin{equation}
  \label{eq:adjoint_stability}
    \| p \|_{\polH^{s-\theta}(\Omega)}
    \lesssim
    \| f - \fraka(\cdot,0) \|_{L^2(\Omega)} + \| q \|_\Ldeux
    +
    \sum_{\vertex \in \mathcal{D}} |u_\vertex|.
  \end{equation}
\end{theorem}
\begin{proof}
 The proof follows directly from applying \cref{thm:BNB}. Note that in the notation of that result:
 \begin{enumerate}[leftmargin=*]
  \item $c \coloneqq \tfrac{\partial \fraka}{ \partial u} (\cdot,u) \in L^{\infty}(\Omega)$ because for $q \in \polQ_{ad}$, we have, according to \cref{thm:regularity_fractional_semilinear_PDE}, $u = \mathcal{S}q \in C(\bar \Omega)$. Consequently,
  there exists $\mathfrak{m}>0$ such that $|u(x)| \leq \mathfrak{m}$ for all $x \in \bar{\Omega}$. This, together with \Cref{A3}, gives the boundedness of $c$.

  \item $ c \geq 0$. This follows from \Cref{A2}.

 \item $\mu \coloneqq \sum_{\vertex \in \mathcal{D}} (u(\vertex) - \mathfrak{u}_\vertex) \delta_\vertex \in \mathbb{H}^{-s-\theta}(\Omega)$ and
 \[
  \| \mu \|_{\mathbb{H}^{-s-\theta}(\Omega)} = \sup_{v \in \mathbb{H}^{s+\theta}(\Omega)} \frac{\langle \mu, v \rangle}{ \| v \|_{\mathbb{H}^{s+\theta}(\Omega)}}
  \leq
  \| \mu \|_{\mathcal{M}(\Omega)} \sup_{v \in \mathbb{H}^{s+\theta}(\Omega)} \frac{\| v \|_{C(\bar \Omega)} }{ \| v \|_{\mathbb{H}^{s+\theta}(\Omega)}}
  \lesssim \| \mu \|_{\mathcal{M}(\Omega)},
 \]
  where we have used the continuous embedding $\mathbb{H}^{s+\theta}(\Omega) \hookrightarrow C(\bar \Omega)$. Recall that $1 < s + \theta < 2s$ (see \cref{eq:s_and_theta}).

  \item Finally,
  \begin{align*}
    \| \mu\|_{\mathcal{M}(\Omega)} &\leq \sum_{\vertex \in \calD} |u(\vertex) - u_\vertex| \lesssim \| u \|_{C(\bar\Omega)} + \sum_{\vertex \in \calD} | u_\vertex|
    \\
    &\lesssim \| f - \fraka(\cdot,0) \|_{L^2(\Omega)} + \| q \|_\Ldeux + \sum_{\vertex \in \mathcal{D}} |u_\vertex|
  \end{align*}
  where, in the last step, we used \cref{eq:regularity_semilinear_PDE}.
 \end{enumerate}
\end{proof}

We are now ready to derive the first order optimality conditions.

\begin{theorem}[first order optimality conditions]
\label{thm:first_order_optimality_conditions}
Every locally optimal control $\bar q \in \mathbb{Q}_{ad}$ satisfies the variational inequality
 \begin{equation}
  \label{eq:first_order}
    \int_{\Omega} (\bar{p} + \alpha \bar{q}) (q -\bar{q})\mathrm{d}x \geq 0
    \quad
    \forall q \in \mathbb{Q}_{ad},
  \end{equation}
where $\bar{p} \in \mathbb{H}^{s-\theta}(\Omega)$ is the unique solution to the adjoint problem \cref{eq:adjoint_equation_weak}, with $u$ replaced by $\bar{u} = \mathcal{S}\bar{q}$.
\end{theorem}
\begin{proof}
The proof builds on and extends the arguments developed for the simpler linear case $\fraka \equiv 0$ in \cite[Theorem 4.2]{Fractional_Delta}. We begin by rewriting inequality \cref{eq:first_order_basic} as follows:
\[
 0 \leq j'(\bar q)(q - \bar{q}) =  \sum_{\vertex \in \mathcal{D}}   (\calS \bar q (\vertex) - u_\vertex) \calS'(\bar q) (q - \bar q)(\vertex) + \int_{\Omega} \alpha \bar{q} (q - \bar{q}) \mathrm{d}x
 \eqqcolon \mathrm{I} + \mathrm{II}
\]
for all $q \in \mathbb{Q}_{ad}$. The term $\mathrm{II}$ is already present in the desired inequality \cref{eq:first_order}, so we focus on term $\mathrm{I}$. Let $q \in \mathbb{Q}_{ad}$, and set $u = \mathcal{S} q$. Define $\chi \coloneqq \calS'(\bar{q})(q-\bar{q})$ and note that \cref{thm:differentiability} shows that $\chi$ solves problem \cref{eq:first_derivative} with $u$ replaced by $\bar{u} = \mathcal{S}\bar{q}$ and $w$ replaced by $q-\bar{q} \in L^2(\Omega)$. We now use the arguments presented in the proof of \cref{thm:regularity_fractional_semilinear_PDE} to deduce that
\begin{equation}
\label{eq:chi_regular}
 \chi \in \mathbb{H}^{2s}(\Omega) \cap C(\bar \Omega),
 \qquad
\| \chi \|_{\mathbb{H}^{2s}(\Omega)} + \| \chi \|_{C(\bar \Omega)} \lesssim \left( 1+C_{\mathfrak{m}} \right) \| q-\bar{q} \|_{L^2(\Omega)},
\end{equation}
where, to bound $\tfrac{\partial \fraka}{\partial u}(\cdot,\bar{u})$ and thus obtain the estimate in \cref{eq:chi_regular}, we have used assumption \Cref{A3} with $\mathfrak{m} = \| \bar{u} \|_{L^{\infty}(\Omega)}$. Since $1 < s + \theta < 2s$ (see the second estimate in \cref{eq:s_and_theta}), we have $\mathbb{H}^{2s}(\Omega) \hookrightarrow \mathbb{H}^{s+ \theta}(\Omega)$. As a result, the function $\chi = S'(\bar{q})(q-\bar{q})$ belongs to $\mathbb{H}^{s+ \theta}(\Omega)$ and is therefore an admissible test function for the adjoint problem \cref{eq:adjoint_equation_weak}. Setting $v = \chi$ gives
\begin{equation}
\label{eq:aux_0_first_order}
\calA(\bar{p}, \chi ) + \int_{\Omega} \frac{\partial \fraka}{\partial u}(\cdot,\bar{u}) \bar{p} \chi \mathrm{d}x
=
\sum_{\vertex \in \mathcal{D}} (\bar{u}(\vertex) - u_\vertex) \calS'(\bar{q})(q-\bar{q})(\vertex) = \mathrm{I}.
\end{equation}
On the other hand, we would like to set $\bar{p}$ as a test function in the problem that $\chi = \calS'(\bar q) (u-\bar{u})$ solves. If that were possible, we would obtain
\begin{equation}
\label{eq:aux_1_first_order}
 \mathcal{B}(\chi,\bar{p}) + \int_{\Omega} \frac{\partial \fraka}{\partial u}(\cdot,\bar{u}) \chi \bar{p} \mathrm{d}x = \int_{\Omega} (q -\bar{q}) \bar{p} \mathrm{d}x.
\end{equation}
However, this is not possible because $\bar{p} \in \polH^{s-\theta}(\Omega) \setminus \polH^s(\Omega)$, so \cref{eq:aux_1_first_order} must be justified by different means. Therefore, we follow an argument from the proof of \cite[Theorem 4.2]{Fractional_Delta} and let $\{ p_k \}_{k \in \Natural}$ be a sequence in $C_0^{\infty}(\Omega)$ such that $p_k \rightarrow \bar{p}$ in $\mathbb{H}^{s-\theta}(\Omega)$ as $k \uparrow \infty$. Taking advantage of the fact that, for each $k \in \Natural$, $p_k$ belongs to $\polH^s(\Omega)$ we set $v = p_k$ in the problem that $\chi = S'(\bar q) (q-\bar{q})$ solves to obtain
\begin{equation}
\label{eq:aux_2_first_order}
 \mathcal{B}(\chi,p_k) + \int_{\Omega} \frac{\partial \fraka}{\partial u}(\cdot,\bar{u}) \chi p_k \mathrm{d}x = \int_{\Omega} (q -\bar{q}) p_k \mathrm{d}x
 \quad
 \forall k \in \Natural.
\end{equation}
Since $\chi \in \mathbb{H}^{2s}(\Omega)$ and $\mathbb{H}^{2s}(\Omega) \hookrightarrow \mathbb{H}^{s+ \theta}(\Omega)$, we have $\mathcal{B}(\chi,p_k) = \mathcal{A}(p_k,\chi)$. We now use the fact that $\mathcal{A}$ is continuous in $\polH^{s-\theta}(\Omega) \times \polH^{s+\theta}(\Omega)$ to deduce that
\begin{equation}
\label{eq:aux_3_first_order}
\lim_{k \uparrow \infty} \mathcal{B}(\chi,p_k)
=
\lim_{k \uparrow \infty}  \mathcal{A}(p_k,\chi)
=
\mathcal{A}(\bar{p},\chi).
\end{equation}
We also note that, using \Cref{A3} with $\frakm = \| \bar{u} \|_{L^\infty(\Omega)}$,
\begin{equation}
\label{eq:aux_4_first_order}
 \int_{\Omega} \left| \frac{\partial \fraka}{\partial u}(\cdot,\bar{u}) \chi (\bar{p} - p_k) \right|\mathrm{d}x \leq C_{\mathfrak{m}} \| \chi \|_{L^2(\Omega)} \| \bar{p} - p_k \|_{L^2(\Omega)} \rightarrow 0,
 \quad
 k \uparrow \infty.
\end{equation}
We thus take the limit in \cref{eq:aux_2_first_order} as $k \uparrow \infty$ and use \cref{eq:aux_3_first_order} and \cref{eq:aux_4_first_order} to deduce that
\[
 \mathcal{A}(\bar{p},\chi) + \int_{\Omega} \frac{\partial a}{\partial u}(\cdot,\bar{u}) \chi \bar{p} \mathrm{d}x = \int_{\Omega} (q -\bar{q}) \bar{p} \mathrm{d}x.
\]
From this identity and \cref{eq:aux_0_first_order}, we immediately deduce that
\[
  \mathrm{I} = \int_{\Omega} \bar{p}(q-\bar{q}) \mathrm{d}x.
\]
This yields the desired variational inequality \cref{eq:first_order} and concludes the proof.
\end{proof}

\subsection{Second order optimality conditions}
In this section, we derive necessary and sufficient second order optimality conditions for the optimal control problem \cref{eq:min,eq:state_equation_weak}. To do this, we require that \Cref{A1,A2,A3} hold. In addition, we assume that $f$ and $\fraka(\cdot,0)$ belong to $L^2(\Omega)$.

\subsubsection{Properties of the second derivative of $j$}

We begin our analysis with the following result.

\begin{proposition}[characterization of $j''$ and a Lipschitz property for $j''$]
If \Cref{A1,A2,A3} hold, then, for every $q, w_1, w_2 \in L^2(\Omega)$, we have the characterization
 \begin{equation}
 \label{eq:jpp_characterization}
  j''(q)(w_1,w_2) = \alpha (w_1,w_2)_{L^2(\Omega)} - \int_{\Omega} \frac{\partial^2 \fraka}{\partial u^2}(\cdot,u) \phi_{w_1} \phi_{w_2} p \mathrm{d}x + \sum_{\vertex \in \mathcal{D}} \phi_{w_1}(\vertex) \phi_{w_2}(\vertex),
 \end{equation}
where $u = \mathcal{S}q$, $p$ is the solution to \cref{eq:adjoint_equation_weak}, and $\phi_{w_i} = S'(q)w_i$, with $i \in \{1,2\}$. In addition, if $q_1, q_2 \in \mathbb{Q}_{ad}$, then there exists a constant $\mathfrak{C}$ such that
\begin{equation}
\label{eq:jpp_Lipschitz}
|j''(q_1)(w,w) - j''(q_2)(w,w)| \leq \mathfrak{C} \| q_1 - q_2 \|_{L^2(\Omega)} \| w \|_{L^2(\Omega)}^2.
 \end{equation}
\end{proposition}
\begin{proof}
Since $j$ is of class $C^2$ (see \cref{pro:j_is_of_class_C2}), we perform simple computations to conclude that for every $q,w_1,w_2 \in L^2(\Omega)$, we have
\begin{equation}
\label{eq:jpp_aux}
  j''(q)(w_1,w_2) = \sum_{\vertex \in \mathcal{D}} \left[ \phi_{w_2}(\vertex) \phi_{w_1}(\vertex) + (Sq(\vertex) - u_\vertex) \psi(\vertex) \right] + \alpha (w_1,w_2)_{L^2(\Omega)},
\end{equation}
where $\psi = S''(q)(w_1,w_2)$. Since $\psi \in \mathbb{H}^{2s}(\Omega)$ and $\mathbb{H}^{2s}(\Omega) \hookrightarrow \mathbb{H}^{s+\theta}(\Omega)$, we can set $v = \psi$ in the adjoint problem \cref{eq:adjoint_equation_weak}. On the other hand, we can use a similar approximation argument as in the proof of \cref{thm:first_order_optimality_conditions}, which essentially allows us to set $v = p$ in problem \cref{eq:second_derivative}. From the resulting relations, we can deduce that
\[
 \sum_{\vertex \in \mathcal{D}} (Sq(\vertex) - u_\vertex) \psi(\vertex)
 =
 - \int_{\Omega} \frac{\partial^2 \fraka}{\partial u^2}(\cdot,u) \phi_{w_1} \phi_{w_2} p \mathrm{d}x.
\]
By applying this identity in \cref{eq:jpp_aux}, we obtain \cref{eq:jpp_characterization}.

We now derive the bound \cref{eq:jpp_Lipschitz}. Let $q_1, q_2 \in\mathbb{Q}_{ad}$, and let $u_i = \mathcal{S} q_i$, with $i \in \{1,2\}$. Define
\[
\frakd(q_1,q_2) \coloneqq j''(q_1)(w,w) - j''(q_2)(w,w).
\]
We use the characterization \cref{eq:jpp_characterization} for $j''$ and write
\begin{multline}
\label{eq:dq1q2}
\mathfrak{d}(q_1,q_2) = \int_{\Omega} \left[-\frac{\partial^2 \fraka}{\partial u^2}(\cdot,u_1) \varphi^2 p_1
 +
\frac{\partial^2 \fraka}{\partial u^2}(\cdot,u_2) \chi^2 p_2 \right] \mathrm{d}x
\\
+
\sum_{\vertex \in \mathcal{D}} [\varphi^2(\vertex) - \chi^2(\vertex)]
\eqqcolon \mathfrak{J} + \mathfrak{K},
\end{multline}
where $\varphi = \mathcal{S}'(q_1) w$, $\chi = \mathcal{S}'(q_2) w$, and $p_i$ is the solution to \cref{eq:adjoint_equation_weak}, with $u$ replaced by $u_i$. Here, $i \in \{1,2\}$. To bound $\mathfrak{d}(q_1,q_2)$, we construct suitable differences as follows:
\begin{multline}
\label{eq:mathfrakJ}
 \mathfrak{J}
 =
 \int_{\Omega} \left[
 \frac{\partial^2 \fraka}{\partial u^2}(\cdot,u_2)
 -
 \frac{\partial^2 \fraka}{\partial u^2}(\cdot,u_1)
 \right]
 \chi^2 p_2 \mathrm{d}x
 + \int_{\Omega}
 \frac{\partial^2 \fraka}{\partial u^2}(\cdot,u_1)\chi^2 ( p_2 -p_1) \mathrm{d}x
 \\
 +\int_{\Omega} \frac{\partial^2 \fraka}{\partial u^2}(\cdot,u_1) (\chi^2 - \varphi^2)  p_1 \mathrm{d}x \eqqcolon \mathfrak{J}_1 + \mathfrak{J}_2 + \mathfrak{J}_3.
\end{multline}
In what follows, we bound the terms $\mathfrak{J}_1$, $\mathfrak{J}_2$, and $\mathfrak{J}_3$. To control the term $\mathfrak{J}_1$, we use H{\"o}lder's inequality, \Cref{A3}, the Lipschitz property \cref{eq:Lipschitz_property}, and the regularity estimates $\| \chi  \|_{C(\bar \Omega)} \lesssim  \| \chi  \|_{\mathbb{H}^{2s}(\Omega)} \lesssim  \| w  \|_{L^{2}(\Omega)}$:
\begin{equation}
\label{eq:estimate_J1}
  \begin{aligned}
    \mathfrak{J}_1
    &\leq
    C_{\mathfrak{m}}  \| u_1 - u_2  \|_{L^{\infty}(\Omega)} \| \chi \|^2_{L^{\infty}(\Omega)} \| p_2 \|_{L^1(\Omega)}
    \\
    &\lesssim
    \| q_1 - q_2  \|_{L^{2}(\Omega)} \| w  \|^2_{L^{2}(\Omega)}  \| p_2 \|_{L^1(\Omega)}.
 \end{aligned}
\end{equation}
Here, $\mathfrak{m} = \max \{ \| u_1  \|_{L^{\infty}(\Omega)}, \| u_2  \|_{L^{\infty}(\Omega)}\}$. To bound the right-hand side of \cref{eq:estimate_J1}, we now use the following stability bounds, which follow from \Cref{thm:adjoint_well-posedness}:
\begin{equation}
\| p_2 \|_{L^1(\Omega)} \lesssim \| p_2 \|_{\mathbb{H}^{s-\theta}(\Omega)}
\lesssim
\| f - \fraka(\cdot,0) \|_{L^2(\Omega)} + \| q_2 \|_\Ldeux
+
\sum_{\vertex \in \mathcal{D}} |\mathfrak{u}_\vertex|.
\label{eq:stability_p_2}
\end{equation}
This, combined with \cref{eq:estimate_J1}, gives the following bound for $\frakI_1$:
\[
  \frakI_1 \lesssim \| q_1 - q_2  \|_{L^{2}(\Omega)} \| w  \|^2_{L^{2}(\Omega)}.
\]
The implicit constant depends on the control problem data but is independent of $q_1$, $q_2$, and $w$. Recall that $q_1, q_2 \in \mathbb{Q}_{ad}$. The control of the term $\mathfrak{J}_2$ follows similar arguments:
\begin{equation}
\label{eq:estimate_J2}
 \mathfrak{J}_2
 \leq
 C_{\mathfrak{m}} \| \chi \|^2_{L^{\infty}(\Omega)} \| p_2 - p_1 \|_{L^1(\Omega)}
 \\
 \lesssim
\| w  \|^2_{L^{2}(\Omega)} \| p_2 - p_1 \|_{L^1(\Omega)},
\end{equation}
where $\mathfrak{m} = \| u_1  \|_{L^{\infty}(\Omega)}$. To bound $\| p_2 - p_1 \|_{L^1(\Omega)}$, we note that $\wp \coloneqq p_2 - p_1 \in \mathbb{H}^{s-\theta}(\Omega)$ satisfies
\begin{multline*}
  \calA(\wp,v) + \int_\Omega \frac{\partial \fraka}{\partial u}(\cdot, u_2) \wp v \diff x =\int_\Omega \left[ \frac{\partial \fraka}{\partial u}(\cdot, u_1) - \frac{\partial \fraka}{\partial u}(\cdot, u_2) \right] p_1 v \diff x
  \\
  + \sum_{\vertex \in \calD} (u_2(\vertex) - u_1(\vertex)) \langle \delta_\vertex , v \rangle
  \quad \forall v \in \polH^{s+\theta}(\Omega).
\end{multline*}
Since $s - \theta > 0$, it is clear that $\mathbb{H}^{s- \theta}(\Omega) \hookrightarrow L^1(\Omega)$. As a result, $\| \wp \|_{L^1(\Omega)}
 \lesssim
 \| \wp \|_{\mathbb{H}^{s-\theta}(\Omega)}$.
We now apply the stability bound from \cref{thm:BNB} to arrive at
\begin{equation*}
 \| \wp \|_{\mathbb{H}^{s-\theta}(\Omega)}
 \lesssim
 \left( C_{\mathfrak{m}} \| p_1 \|_{L^{1}(\Omega)}  + 1 \right) \| u_1 - u_2 \|_{C(\bar \Omega)}.
\end{equation*}
Notice that, to obtain this estimate, the bound
\begin{align*}
  \left\| \left[ \frac{\partial \fraka}{\partial u}(\cdot, u_1) - \frac{\partial \fraka}{\partial u}(\cdot, u_2) \right] p_1 \right\|_{\calM(\Omega)} &\leq
  \left\| \left[ \frac{\partial \fraka}{\partial u}(\cdot, u_1) - \frac{\partial \fraka}{\partial u}(\cdot, u_2) \right] p_1 \right\|_{L^1(\Omega)}
  \\
  &\leq C_\frakm \| u_1 - u_2 \|_{L^\infty(\Omega)} \| p_1 \|_{L^1(\Omega)},
\end{align*}
which follows from \Cref{A3}, was used. Finally, using the Lipschitz property \cref{eq:Lipschitz_property} and a stability bound for $p_1$ in $\mathbb{H}^{s-\theta}(\Omega)$, similar to the one derived for $p_2$ in \cref{eq:stability_p_2}, we conclude that
\[
  \| p_2 - p_1 \|_{\mathbb{H}^{s-\theta}(\Omega)} \lesssim \| q_1 - q_2  \|_{L^{2}(\Omega)}.
\]
By substituting this bound into \cref{eq:estimate_J2}, we finally obtain the control of $\frakJ_2$:
\[
  \frakJ_2 \lesssim \| q_1 - q_2  \|_{L^{2}(\Omega)} \|w \|_\Ldeux^2.
\]

The control of $\mathfrak{J}_3$ is as follows:
\begin{equation}
\label{eq:estimate_J3}
\begin{aligned}
  \mathfrak{J}_3 &\leq C_{\mathfrak{m}} \| \chi - \varphi \|_{L^{\infty}(\Omega)} \left( \| \chi \|_{L^{\infty}(\Omega)} + \| \varphi \|_{L^{\infty}(\Omega)} \right) \| p_1 \|_{L^{1}(\Omega)}
  \\
  &\lesssim \| \chi - \varphi \|_{L^{\infty}(\Omega)} \| w \|_{L^{2}(\Omega)},
\end{aligned}
\end{equation}
where $\mathfrak{m} = \| u_1  \|_{L^{\infty}(\Omega)}$. To obtain the last bound, we used the regularity estimates $\| \chi  \|_{C(\bar \Omega)} \lesssim  \| \chi  \|_{\mathbb{H}^{2s}(\Omega)} \lesssim  \| w  \|_{L^{2}(\Omega)}$ and $\| \varphi  \|_{C(\bar \Omega)} \lesssim  \| \varphi  \|_{\mathbb{H}^{2s}(\Omega)} \lesssim  \| w  \|_{L^{2}(\Omega)}$. The control of $\| p_1 \|_{L^{1}(\Omega)}$ is similar to that derived for $p_2$ in \cref{eq:stability_p_2}. Thus, it suffices to bound $\| \chi - \varphi \|_{L^{\infty}(\Omega)}$. To do this, we note that the function $\Phi \coloneqq \chi - \varphi \in \mathbb{H}^s(\Omega)$ satisfies, for every $v \in \mathbb{H}^s(\Omega)$,
\[
  \calB( \Phi ,v) + \int_{\Omega} \frac{\partial \fraka}{\partial u}(\cdot,u_2)  \Phi v \mathrm{d}x =
  \int_{\Omega} \left[ \frac{\partial \fraka}{\partial u}(\cdot,u_1) - \frac{\partial \fraka}{\partial u}(\cdot,u_2) \right]  \varphi v \mathrm{d}x,
\]
which, upon setting
\[
  \frakb(x,\zeta) \coloneqq \frac{\partial \fraka}{\partial u}(x,u_2(x)) \zeta,
\]
fits into the setting of problem \cref{eq:Fractional_semilinear_PDE_weak}. Indeed, monotonicity follows from \Cref{A2}, while \cref{eq:a_bounded_for_psi_m} is a consequence of \Cref{A3} and $u_2 \in L^\infty(\Omega)$. Another application of \Cref{A3} then allows us to deduce the Lipschitz property \cref{eq:a_is_Lipschitz}. Thus, we invoke \cref{eq:regularity_semilinear_PDE} to conclude
\[
  \| \chi - \varphi \|_{C(\bar\Omega)} \lesssim C_{\mathfrak{m}} \| u_1 - u_2 \|_{L^2(\Omega)} \| \varphi \|_{L^{\infty}(\Omega)} \lesssim \| q_1 - q_2 \|_{L^2(\Omega)} \| w  \|_{L^{2}(\Omega)}.
\]
Now substitute
this estimate into \cref{eq:estimate_J3}, to obtain $\mathfrak{J}_3 \lesssim \| q_1 - q_2 \|_{L^2(\Omega)} \| w  \|^2_{L^{2}(\Omega)}$.

The estimate for $\mathfrak{J}$ thus follows from substituting the bounds derived for $\mathfrak{J}_1$, $\mathfrak{J}_2$, and $\mathfrak{J}_3$ into \cref{eq:mathfrakJ}; that is, we have
\[
  \mathfrak{J} \lesssim \| q_1 - q_2 \|_{L^2(\Omega)} \| w  \|^2_{L^{2}(\Omega)}.
\]

The final step is to control the term $\mathfrak{K}$ in \cref{eq:dq1q2}:
\[
 \mathfrak{K} \lesssim \| \varphi - \chi \|_{L^{\infty}(\Omega)} \left( \| \varphi \|_{L^{\infty}(\Omega)} + \| \chi \|_{L^{\infty}(\Omega)} \right)
 \lesssim
 \| q_1 - q_2 \|_{L^2(\Omega)} \| w  \|^2_{L^{2}(\Omega)}.
\]

Collect the estimates for $\mathfrak{J}$ and $\mathfrak{K}$ and use them in \cref{eq:dq1q2} to deduce the desired bound. This concludes the proof.
\end{proof}

\subsubsection{Second order necessary optimality conditions}
We begin this section by introducing some preliminary material. Let $(\bar u, \bar q, \bar p)$ $\in$ $\mathbb{H}^s(\Omega) \times \mathbb{Q}_{ad} \times \mathbb{H}^{s-\theta}(\Omega)$ satisfy the first order optimality conditions \cref{eq:state_equation_weak,eq:adjoint_equation_weak,eq:first_order}. Define
\[
  \bar{\mathsf{d}} \coloneqq \bar{p} + \alpha \bar{q}.
\]
From the variational inequality \cref{eq:first_order}, it can be deduced that, for almost every $x \in \Omega$,
\begin{equation}
\label{eq:mathsfd_ineq}
  \begin{dcases}
    \bar{\mathsf{d}}(x) = 0, & \bar{q}(x) \in (a,b),
   \\
   \bar{\mathsf{d}}(x) \geq 0, & \bar{q}(x)=a,
   \\
   \bar{\mathsf{d}}(x) \leq 0, & \bar{q}(x)=b.
  \end{dcases}
\end{equation}

We now define the so-called cone of critical directions. To do this, we introduce
\begin{align*}
  \Upsilon_{\bar q} &\coloneqq \left\{ h \in \Ldeux \ : \ \bar{\mathsf{d}}(x) \neq 0 \implies h(x) = 0 \ \mae \ x \in \Omega \right\},
  \\
  \pluto_{\bar q}^a &\coloneqq \left\{ h \in \Ldeux \ : \ \bar{q}(x) =  a \implies h(x) \geq 0 \ \mae \ x \in \Omega \right\},
  \\
  \pluto_{\bar q}^b &\coloneqq \left\{ h \in \Ldeux \ : \ \bar{q}(x) = b \implies h(x) \leq 0  \ \mae \ x \in \Omega \right\},
  \\
  \pluto_{\bar q} &\coloneqq \pluto_{\bar q}^a \cap \pluto_{\bar q}^b.
\end{align*}
Notice that although the definition of $\Upsilon_{\bar q}$ involves the function $\bar{\mathsf{d}}$, this function is uniquely determined by $\bar{q}$. The cone of critical directions is then defined as
\begin{equation}
\label{eq:cone_of_critical}
  C_{\bar{q}} \coloneqq \Upsilon_{\bar q } \cap \pluto_{\bar q}.
\end{equation}

In the following result, we present necessary second order optimality conditions. The proof uses minor adaptations of the arguments used to obtain \cite[Theorem 23]{MR3586845}. As the proof is brief, we include it for completeness.

\begin{theorem}[second order necessary optimality conditions]
\label{thm:necessary_second_order}
If $\bar{q} \in \mathbb{Q}_{ad}$ is a locally optimal control for \cref{eq:min,eq:state_equation_weak}, then $j''(\bar q)(h,h) \geq 0$ for all $h \in C_{\bar{q}}$.
\end{theorem}
\begin{proof}
Let $h \in C_{\bar q}$. Since $b-a>0$ there is $K \in \Natural$ such that, for every $k \geq K$,
\[
  a < a+ \frac1k < b - \frac1k < b.
\]
Now, for each $k \geq K$, we define the function $h_k : \Omega \rightarrow \Real$ as
\begin{equation}
\label{eq:hk}
  h_k(x) \coloneqq
  \begin{dcases}
    0, & \bar{q}(x) \in \left( a, a+\frac1k \right) \cup \left( b-\frac1k, b \right) ,
    \\
    \min\left\{ k, \max\{ -k, h(x) \} \right\}, & \text{otherwise}.
  \end{dcases}
\end{equation}
Since $h \in C_{\bar{q}}$,  it follows directly from the definition of $h_k$ that $h_k$ also belongs to $C_{\bar{q}}$. In addition, for almost every $x\in \Omega$, $| h_k(x) | \leq | h(x) |$ and $h_k(x) \rightarrow h(x)$ as $k \uparrow \infty$. An application of the dominated convergence theorem shows that, as $k \uparrow \infty$, $h_k \rightarrow h$ in $L^2(\Omega)$. We now define $\rho_{\star} \coloneqq \min\{k^{-2}, (b-a)k^{-1}\}$. It can be proved that $\bar{q} + \rho h_k $ belongs to $\mathbb{Q}_{ad}$ for every $k \geq K$ and $\rho \in (0,\rho_{\star}]$.

We now use the fact that, for $\rho \in (0,\rho_\star]$, $\bar{q} + \rho h_k$ is admissible and that $\bar{q}$ is a local minimizer to deduce that, after possibly reducing $\rho$, $j(\bar{q}) \leq j(\bar{q} + \rho h_k)$. The next step is to apply Taylor's theorem and use the fact that $j'(q) h_k = 0$, which follows from $h_k \in C_{\bar{q}}$ and \cref{eq:mathsfd_ineq}, to obtain
\[
 0 \leq j(\bar{q} + \rho h_k)  - j(\bar{q}) = \rho j'(\bar{q})h_k + \frac{\rho^2}{2} j''(\bar{q} + \theta_k \rho h_k)(h_k,h_k) = \frac{\rho^2}{2} j''(\bar{q} + \theta_k \rho h_k)(h_k,h_k),
\]
where $\theta_k \in (0,1)$. Multiply both sides of the previous inequality by $2/\rho^2$, then take the limit as $\rho \downarrow 0$ and use \cref{eq:jpp_Lipschitz} to obtain $j''(\bar{q})(h_k, h_k) \geq 0$. The final step is to use the characterization \cref{eq:jpp_characterization}, well-posedness and regularity results for problem \cref{eq:first_derivative}, and the fact that $h_k \rightarrow h$ in $L^2(\Omega)$ as $k \uparrow \infty$ to deduce that $j''(\bar{q})(h,h) \geq 0$. This concludes the proof.
\end{proof}

\subsubsection{Second order sufficient optimality conditions}
We now establish sufficient second order optimality conditions with a minimal gap with respect to the necessary conditions derived in \cref{thm:necessary_second_order}. The proof presented here, based on an adaptation of the arguments from the proof of \cite[Theorem 23]{MR3586845}, emphasizes the role of the fractional operator $(-\Delta)^s$ and the inclusion of point evaluations of the state in the cost functional.

\begin{theorem}[second order sufficient  optimality conditions]
Assume that the triplet $(\bar u, \bar q, \bar p)$ $\in$ $\mathbb{H}^s(\Omega) \times \mathbb{Q}_{ad} \times \mathbb{H}^{s-\theta}(\Omega)$ satisfies the first order optimality conditions \cref{eq:state_equation_weak,eq:adjoint_equation_weak,eq:first_order}, and that $j''(\bar{q})(v,v) > 0$ for all $v \in C_{\bar{q}} \setminus \{ 0 \}$. Then, there exist $\mu>0$ and $\sigma>0$ such that, for every $q \in \mathbb{Q}_{ad}$ that satisfies
\[
  \| q - \bar{q} \|_{L^2(\Omega)} \leq \sigma,
\]
we have
\begin{equation}
\label{eq:quadratic_grow}
 j(q) \geq j(\bar{q}) + \frac{\mu}{2} \| q - \bar{q} \|^2_{L^2(\Omega)}
 .
\end{equation}
\end{theorem}
\begin{proof}
We proceed by contradiction and assume that \cref{eq:quadratic_grow} does not hold; that is, for each $k \in \Natural$, there exists $q_k \in \mathbb{Q}_{ad}$ such that
\begin{equation}
\label{eq:contradiction}
 j(q_k) < j(\bar{q}) + \frac{1}{2k} \| q_k - \bar{q} \|^2_{L^2(\Omega)},
 \qquad
\| q_k - \bar{q} \|_{L^2(\Omega)} < \frac{1}{k}.
\end{equation}
For each $k \in \Natural$, we define $\rho_k \coloneqq \| q_k - \bar{q} \|_{L^2(\Omega)}$ and $h_k \coloneqq \rho_k^{-1}(q_k - \bar{q})$. It is clear that $\| h_k \|_{L^2(\Omega)} =1 $ for all $k \in \Natural$. We can thus conclude the existence of a nonrelabeled subsequence $\{ h_k \}_{k \in \Natural}$ such that $h_k \rightharpoonup h$ in $L^2(\Omega)$ as $k \uparrow \infty$.  We now proceed in several steps.

\emph{Step 1: $h \in C_{\bar{q}}$}. First, note that each $h_k \in \pluto_{\bar q}$. Since $\pluto_{\bar q}$ is a closed and convex set in $\Ldeux$ and, as $k \uparrow \infty$, we have $h_k \rightharpoonup h$ in $L^2(\Omega)$, we deduce that $h \in \pluto_{\bar q}$ as well. We now prove that $h \in \Upsilon_{\bar q}$. To do this, for each $k \in \Natural$, we use a mean value theorem on the function $\theta \mapsto j( \bar{q} + \theta(q_k - \bar{q} ) )$ to obtain, for some $\theta_k \in (0,1)$,
\[
  j(q_k) - j(\bar{q}) = j'\left( \tilde{q}_k \right) (q_k - \bar{q}), \qquad \tilde{q}_k = \bar{q} + \theta_k(q_k - \bar{q}).
\]
Then, using \cref{eq:contradiction}, we deduce that
\[
  j'(\tilde{q}_k) h_k = \frac1{\rho_k} j'\left( \tilde{q}_k \right) (q_k - \bar{q}) = \frac{ j(q_k) - j(\bar{q}) }{\rho_k} < \frac{\rho_k}{2k}.
\]
Since the sequence $\{\rho_k\}_{k \in \Natural}$ is bounded, we conclude that
\[
  \limsup_{k \uparrow \infty} j'(\tilde{q}_k) h_k \leq \lim_{k \uparrow \infty} \frac{\rho_k}{2k} = 0.
\]
We now compute
\[
  j'(\tilde{q}_k) h_k = \int_{\Omega} (\tilde{p}_k + \alpha \tilde{q}_k) h_k \mathrm{d}x,
\]
where $\tilde{p}_k$ denotes the solution to \cref{eq:adjoint_equation_weak} with $u$ replaced by $\tilde{u}_k \coloneqq \calS \tilde{q}_k$. We claim that $\tilde{p}_k \to \bar{p}$ in $\polH^{s-\theta}(\Omega)$ as $k \uparrow \infty$. To see this, we first note that an immediate application of the Lipschitz property \cref{eq:Lipschitz_property} shows that
\begin{equation}
\label{eq:convergence_u}
\| \bar{u} - \tilde{u}_k \|_{\mathbb{H}^{2s}(\Omega)} + \|\bar{u} - \tilde{u}_k \|_{C(\bar \Omega)} \lesssim \| \bar{q} - \tilde{q}_k \|_{L^2(\Omega)}
\leq \| \bar{q} - q_k \|_{L^2(\Omega)} \rightarrow 0
\end{equation}
as $k \uparrow \infty$. Here, we also used that $\tilde{q}_k = \bar{q} + \theta_k(q_k - \bar{q})$, the fact that $\theta_k \in (0,1)$, and \cref{eq:contradiction}. With this, we write the linear problem that $\bar{p} - \tilde{p}_k$ solves, apply the stability bound from \cref{thm:BNB}, and estimate the corresponding forcing term to obtain
 \begin{equation}
  \label{eq:convergence_p}
  \| \bar{p} - \tilde{p}_k \|_{\mathbb{H}^{s-\theta}(\Omega)} \lesssim
  \|\bar{u} - \tilde{u}_k \|_{C(\bar \Omega)}
  \lesssim
  \| \bar{q} - q_k \|_{L^2(\Omega)} \rightarrow 0
 \end{equation}
as $k \uparrow \infty$.
Since we also have that $\tilde{q}_k \to \bar{q}$ in $\Ldeux$, we conclude that $\mathsf{d}_k \coloneqq \tilde{p}_k + \alpha \tilde{q}_k \rightarrow \bar{p} + \alpha \bar{q} = \bar{\mathsf{d}}$ in $L^2(\Omega)$ as $k \uparrow \infty$. Recall that $s-\theta > 0$. We now use that, as $k \uparrow \infty$, we have $h_k \rightharpoonup h$ in $L^2(\Omega)$, to obtain
\begin{equation}
\label{eq:aux_bound}
 j'(\bar{q})h = \int_{\Omega} \bar{\mathsf{d}}(x) h(x) \mathrm{d}x = \lim_{k \uparrow \infty}  \int_{\Omega} \mathsf{d}_k(x) h_k(x) \mathrm{d}x = \lim_{k \uparrow \infty} j'(\tilde{q}_k)h_k \leq 0.
\end{equation}
We now use the definition of $h_k$, the first order optimality condition \cref{eq:first_order}, and the fact that $q_k \in \polQ_{ad}$ to obtain
\[
  \int_{\Omega} \bar{\mathsf{d}}(x)h_k(x) \mathrm{d}x = \rho_k^{-1} \int_{\Omega} \bar{\mathsf{d}}(x)( q_k(x) - \bar{q}(x) ) \mathrm{d}x \geq 0
\]
for all $k \in \Natural$. Take the limit as $k\uparrow \infty$, and combine with \cref{eq:aux_bound} to deduce that
\[
  \int_{\Omega} \bar{\mathsf{d}}(x) h(x) \mathrm{d}x = 0.
\]
Finally, since $h \in \pluto_{\bar q}$,
\[
  \int_{\Omega} \bar{\mathsf{d}}(x) h(x) \mathrm{d}x = \int_{\Omega} |\bar{\mathsf{d}}(x) h(x)| \mathrm{d}x,
\]
Thus, for almost every $x \in \Omega$ such that $\bar{\mathsf{d}}(x) \neq 0$, we have $h(x) = 0$, and so $h \in C_{\bar q}$.

\emph{Step 2: $h \equiv 0$}. We begin with a straightforward application of Taylor’s theorem to obtain
\[
   j(q_k) - j(\bar{q}) - j'(\bar{q}) (q_k - \bar{q})
 =
 \frac{1}{2}j''(\hat{q}_k) (q_k - \bar{q}, q_k - \bar{q} )
 =
 \frac{\rho_k^2}{2}j''(\hat{q}_k)(h_k,h_k),
\]
where $\hat{q}_k = \bar{q} + \theta_k (q_k - \bar{q})$ and $\theta_k \in (0,1)$. The fact that $q_k \in \polQ_{ad}$, together with \cref{eq:first_order_basic} and \cref{eq:contradiction}, then implies
\[
  \frac{\rho_k^2}{2}j''(\hat{q}_k) (h_k,h_k)
 \leq
 j(q_k) - j(\bar{q})
 < \frac{\rho_k^2}{2k},
\]
and thus
\[
  \limsup_{ k \uparrow \infty} j''(\hat{q}_k) (h_k,h_k) \leq \lim_{k \uparrow \infty} \frac1k = 0.
\]
Our goal is now to show that $j''(\bar q) (h,h) \leq \liminf_{k \uparrow \infty} j''(\hat{q}_k) (h_k,h_k)$. We use the characterization \cref{eq:jpp_characterization} for $j''$ and write
\begin{equation}
\label{eq:jpp_hatqk_hk2}
 j''(\hat{q}_k)(h_k,h_k) = \alpha \| h_k\|^2_{L^2(\Omega)} - \int_{\Omega} \frac{\partial^2 \fraka}{\partial u^2}(\cdot,\hat{u}_k) \phi_{h_k}^2 \hat{p}_k \mathrm{d}x + \sum_{\vertex \in \mathcal{D}} \phi_{h_k}^2(\vertex),
\end{equation}
where $\hat{u}_k \coloneqq \calS \hat{q}_k$, $\hat{p}_k$ denotes the solution to \cref{eq:adjoint_equation_weak} with $u$ replaced by $\hat{u}_k$, and $\phi_{h_k} \coloneqq \calS'(\hat{q}_k) h_k$. Similar arguments as those developed in  \emph{Step 1} show that
\begin{equation}
\label{eq:convergence_hat}
\|\bar{u} - \hat{u}_k \|_{\mathbb{H}^{2s}(\Omega)}
+
\|\bar{u} - \hat{u}_k \|_{C(\bar \Omega)}
+
\| \bar{p} - \tilde{p}_k \|_{\mathbb{H}^{s-\theta}(\Omega)}
  \lesssim
  \| \bar{q} - q_k \|_{L^2(\Omega)} \rightarrow 0,
\end{equation}
as $k \uparrow \infty$. Define $\bar{\phi} \coloneqq \calS'(\bar{q})h$. Since the problem that $\bar{\phi} - \phi_{h_k}$ solves is linear and $h_k \rightharpoonup h$ in $L^2(\Omega)$ as $k\uparrow \infty$, it can be deduced that, as $k \uparrow \infty$,
\begin{equation}
\label{eq:convergece_phi_k}
 \phi_{h_k} \rightharpoonup \bar{\phi}~\textrm{in}~\mathbb{H}^{2s}(\Omega)
 \qquad \implies \qquad
  \phi_{h_k} \rightarrow \bar{\phi}~\textrm{in}~C(\bar\Omega),
\end{equation}
where we used the compact embedding $\polH^{2s}(\Omega) \hookrightarrow C(\bar \Omega)$. Using \cref{eq:convergece_phi_k}, we  immediately deduce that
\[
  \sum_{\vertex \in \mathcal{D}} \phi_{h_k}^2(\vertex) \rightarrow \sum_{\vertex \in \mathcal{D}} \bar{\phi}^2(\vertex) .
\]
As the next step, we estimate
\begin{multline*}
\left|
\int_{\Omega} \left( \frac{\partial^2 \fraka}{\partial u^2}(\cdot,\bar{u}) \bar{\phi}^2 \bar{p}
-
\frac{\partial^2 \fraka}{\partial u^2}(\cdot,\hat{u}_k) \phi_{h_k}^2 \hat{p}_k \right) \mathrm{d}x
\right|
\leq
\left|
\int_{\Omega}  \left( \frac{\partial^2 \fraka}{\partial u^2}(\cdot,\bar{u})
-
\frac{\partial^2 \fraka}{\partial u^2}(\cdot,\hat{u}_k) \right) \bar{\phi}^2 \bar{p} \mathrm{d}x
\right|
\\
+
\left|
\int_{\Omega}
\frac{\partial^2 \fraka}{\partial u^2}(\cdot,\hat{u}_k) (\bar{\phi}^2 - \phi_{h_k}^2) \bar{p} \mathrm{d}x
\right|
+
\left|
\int_{\Omega}
\frac{\partial^2 \fraka}{\partial u^2}(\cdot,\hat{u}_k) \phi_{h_k}^2 (\bar{p} - \hat{p}_k) \mathrm{d}x
\right| \eqqcolon \mathfrak{L}_1^k + \mathfrak{L}_2^k + \mathfrak{L}_3^k.
\end{multline*}
The term $\mathfrak{L}_1^k$ can be estimated using \Cref{A3}, an application of the stability bound \cref{eq:adjoint_stability} to $\bar{p}$, the regularity bounds $\| \bar{\phi} \|_{L^{\infty}(\Omega)} \lesssim \| \bar{\phi} \|_{\mathbb{H}^{2s}(\Omega)} \lesssim \|h\|_{L^2(\Omega)}$,
and \cref{eq:convergence_hat}:
\[
 \mathfrak{L}_1^k \lesssim  \|\bar{u} - \hat{u}_k \|_{C(\bar \Omega)}\|\bar{\phi} \|^2_{L^{\infty}(\Omega)} \|\bar{p} \|_{L^{1}(\Omega)}
 \lesssim
 \| \bar{q} - q_k \|_{L^2(\Omega)} \rightarrow 0,
\]
as $k \uparrow \infty$. We emphasize that the hidden constant in both bounds is independent of $k$. The control of the term $\mathfrak{L}_2^k$ is as follows:
\[
 \mathfrak{L}_2^k \lesssim  \| \bar{\phi} - \phi_{h_k} \|_{L^{\infty}(\Omega)} \left( \| \bar{\phi} \|_{L^{\infty}(\Omega)} + \| \phi_{h_k} \|_{L^{\infty}(\Omega)}\right) \|\bar{p} \|_{L^{1}(\Omega)}
 \lesssim
  \| \bar{\phi} - \phi_{h_k} \|_{L^{\infty}(\Omega)}.
\]
To obtain the last bound, we have used $\| \phi_{h_k} \|_{L^{\infty}(\Omega)} \lesssim \| \phi_{h_k} \|_{\mathbb{H}^{2s}(\Omega)} \lesssim \|h_k\|_{L^2(\Omega)} = 1$. Again, in all the previously derived bounds the hidden constant is independent of $k$. Using \cref{eq:convergece_phi_k}, we deduce that $\mathfrak{L}_{2}^k \rightarrow 0$ as $k \uparrow \infty$. The control of $\mathfrak{L}_{3}^k$ follows similar arguments:
\[
 \mathfrak{L}_3^k \lesssim   \| \phi_{h_k} \|^2_{L^{\infty}(\Omega)}
  \| \bar{p} - \hat{p}_k \|_{L^{1}(\Omega)}
  \lesssim \| \bar{p} - \hat{p}_k \|_{\mathbb{H}^{s-\theta}(\Omega)}
   \lesssim
  \| \bar{q} - q_k \|_{L^2(\Omega)} \rightarrow 0,
\]
as $k \uparrow \infty$. After obtaining these convergence results, we are ready to take the limit in \cref{eq:jpp_hatqk_hk2} as $k \uparrow \infty$. We do this using the fact that $h_k \rightharpoonup h$ in $L^2(\Omega)$ as $k \uparrow \infty$ and the weak lower semicontinuity of $\| \cdot \|^2_{L^2(\Omega)}$:
\[
 \liminf_{k \uparrow \infty} j''(\hat{q}_k) (h_k,h_k) \geq \alpha \| h \|^2_{L^2(\Omega)} - \int_{\Omega} \frac{\partial^2 \fraka}{\partial u^2}(\cdot,\bar u) \bar{\phi}^2 \bar{p} \mathrm{d}x + \sum_{\vertex \in \mathcal{D}} \bar{\phi}^2(\vertex) = j''(\bar{q}) (h,h).
\]
As a result, we have proved that $j''(\bar{q})(h, h) \leq \liminf_{k \uparrow \infty} j''(\hat{q}_k)( h_k,h_k) \leq 0$. Since $h \in C_{\bar{q}}$ and we have assumed that $j''(\bar{q}) (v,v) > 0$ for all $v \in C_{\bar{q}} \setminus \{ 0 \}$, we conclude that $h \equiv 0$.

\emph{Step 3}. Since $h \equiv 0$, it is immediate that $\phi_{h_k} \rightarrow 0$ in $C(\bar \Omega)$ as $k \uparrow \infty$. Let us now use the relation \cref{eq:jpp_hatqk_hk2} and the fact that $\| h_k \|_{L^2(\Omega)} = 1$ for every $k \in \Natural$ to write
\[
 \alpha = \alpha \| h_k \|^2_{L^2(\Omega)} =
 j''(\hat{q}_k)(h_k,h_k)
 +
 \int_{\Omega} \frac{\partial^2 \fraka}{\partial u^2}(\cdot,\hat{u}_k) \phi_{h_k}^2 \hat{p}_k \mathrm{d}x - \sum_{\vertex \in \mathcal{D}} \phi_{h_k}^2(\vertex).
\]
If we take the limit inferior as $k \uparrow \infty$, we obtain $\alpha \leq 0$ because $\liminf_{k \uparrow \infty} j''(\hat{q}_k) h_k^2 \leq 0$. This is a contradiction because $\alpha >0$ and concludes the proof.
\end{proof}

\section*{Funding}

The work of EO has been partially supported by ANID grant FONDECYT 1220156 and by USM through USM  project 2025 PI LII 25 12. The work of AJS has been partially supported by NSF grant DMS-2409918.


\bibliographystyle{siamplain}
\bibliography{biblio}

\end{document}